\documentclass[11pt]{amsart}

\usepackage[left=1in, right=1in, top=1.2in, bottom=1.2in]{geometry}
\usepackage{microtype, mathtools, mathrsfs, enumerate, dsfont, esvect}
\usepackage{tikz,wasysym}
\usepackage{pgf}

\usetikzlibrary{positioning,automata}
\usetikzlibrary{arrows,automata,positioning}
\usepackage[font=small,labelfont=bf]{caption}
\usepackage{bbm}
\usepackage{verbatim}
\usepackage[linktocpage]{hyperref}
\hypersetup{
    colorlinks=true,        
    linkcolor=red,          
    citecolor=blue,         
    filecolor=magenta,      
    urlcolor=cyan           
}

\usepackage{amssymb}
\usepackage{url}

\newcommand{\Z}{\mathbb{Z}}
\newcommand{\Q}{\mathbb{Q}}
\newcommand{\calA}{\mathcal{A}}
\newcommand{\calB}{\mathcal{B}}

\newcommand{\calR}{\mathcal{R}}
\newcommand{\calS}{\mathcal{S}}
\newcommand{\calP}{\mathcal{P}}
\newcommand{\calL}{\mathcal{L}}

\newcommand{\Diam}{\operatorname{Diam}}
\newcommand{\R}{\mathbb{R}}
\newcommand{\N}{\mathbb{N}}

\newtheorem{theoremA}{Theorem}
\newtheorem{thm}{Theorem}

\newtheorem{lem}[thm]{Lemma}
\newtheorem{fact}[thm]{Fact}
\newtheorem{rem}[thm]{Remark}
	
		\numberwithin{thm}{section}
	\newtheorem{lemma}[thm]{Lemma}
	\newtheorem{prop}[thm]{Proposition}
	\newtheorem{cor}[thm]{Corollary}

	\newtheorem*{thm*}{Theorem}
	\newtheorem*{lemma*}{Lemma}
	\newtheorem*{prop*}{Proposition}
	\newtheorem*{cor*}{Corollary}
	\newtheorem*{conj*}{Conjecture}
\theoremstyle{definition}
	
	\newtheorem*{example*}{Example}
	\newtheorem{definition}[thm]{Definition}

\numberwithin{thm}{section}

\theoremstyle{definition}
\newtheorem{defn}[thm]{Definition}

\newtheorem{rmk}[thm]{Remark}
\newcommand{\dcl}{\operatorname{dcl}}

\title[Sparse regular subsets of $\R^n$]{Sparse regular subsets of $\R^n$}

\author{Jason Bell}
\address{University of Waterloo \\
Department of Pure Mathematics \\
Waterloo, Ontario \\
N2L 3G1, Canada}
\email{jpbell@uwaterloo.ca}

\author{Alexi Block Gorman}
\address{The Ohio State University\\
Department of Mathematics\\
Mathematics Tower 100\\ 
231 W 18th Ave\\ 
Columbus, OH, 43210
}
\email{blockgorman.1@osu.edu}

\thanks{The first-named author was supported by NSERC grant RGPIN-2016-03632.  The second-named author was supported by the National Science Foundation under award No. DMS -2303368, as well as the MSCA Cofund MathInGreaterParis program.}


\begin{document}

\title{Sparse regular subsets of the reals}
\date{\today}

\begin{abstract}
This paper concerns the expansion of the real ordered additive group by a predicate for a subset of $[0,1]$ whose base-$r$ representations are recognized by a B\"uchi automaton.
In the case when this predicate is closed, a dichotomy is established for when this expansion is interdefinable with the structure $(\mathbb{R},<,+,0,r^{-\mathbb{N}})$ for some $r \in \mathbb{N}_{>1}$.
In the case when the closure of the predicate has Hausdorff dimension less than $1$, the dichotomy further characterizes these expansions of $(\mathbb{R},<,+,0,1)$ by when they have NIP and NTP$_2$, which is precisely when the closure of the predicate has Hausdorff dimension $0$.
\end{abstract}
\maketitle

\section{Introduction}\label{r:intro}

This paper concerns how geometric and model-theoretic notions of tameness interact for expansions of the real ordered additive group by a subset of $[0,1]$ that is $r$-regular.
Notions from the theory of finite automata are adapted to the real additive group setting, and used to establish a partial tameness dividing line for these expansions.
Key tools include examining the interactions of automaton-theoretic properties and topological properties of subsets of the real numbers.

A finite automaton is a machine with finitely many states, some subset of which are called accepting states, with the remaining states being called rejecting states.  
The machine takes finite-length strings over a fixed finite alphabet $\Sigma$ as input, and, beginning at a fixed starting state, it moves from state to state based upon simple transition rules as it reads the string from left-to-right. 
The machine then either accepts or rejects a string depending upon whether or not it arrives in an accepting state after it has finished reading the word. 

As is standard, if $X$ is a set, then $X^{\omega}$ denotes the set of all sequences of elements of $X$ indexed by $\omega$, or equivalently all functions from the ordinal $\omega$ to the set $X$.
B\"uchi automata differ from classical finite automata in that they take infinite-length strings over $\Sigma$ as input.  This is accomplished by taking a finite automaton and imposing the acceptance condition that an infinite string $w\in \Sigma^{\omega}$ is accepted precisely if a run of the automaton on $w$ enters an accept state infinitely often.

For both kinds of automata, we say that an automaton \emph{recognizes} a set $X$ (either a subset of $\Sigma^*$ or, if it is a B\"uchi automaton, a subset of $\Sigma^{\omega}$) if every element of $X$ is accepted by the automaton, and no element of $X^c$ (the complement) is accepted.
B\"uchi automata have long been an important object of study in logic, combinatorics on words, and theoretical computer science.
Much of the context for the results of this paper and many of the connections between B\"uchi automata and topological phenomena on the reals come from \cite{CLR15}.

We borrow the notion of ``sparse'' sets recognized by automata. These sets arise naturally in many contexts, including: Derksen's \cite{Derksen} extension of the Skolem-Mahler-Lech theorem on zero sets of linear recurrences to positive characteristic base fields; the isotrivial case of the Mordell-Lang conjecture \cite{BGM20, BM19}; the characterization of which $k$-automatic sets $S\subseteq \N$ form an additive basis for the natural numbers \cite{BHS};
Kedlaya's description of the algebraic closure of function fields in positive characteristic \cite{K}, \cite{K2}; as well, as other uses in theoretical computer science \cite{Gawrychowski&Krieger&Rampersad&Shallit:2010, Ibarra&Ravikumar:1986, Trofimov:1981}.

We introduce a natural analog for sparsity in the setting of B\"uchi automata, one which we demonstrate is useful in the setting of real $r$-regular sets.
Among other things, the notion of sparsity for B\"uchi automata allows for the development of a definability dividing line simultaneously in terms of fractal dimension and model-theoretic notions.

To state this result formally requires first defining what it means for a subset of $[0,1]\subseteq \R$ to be $r$-regular.
Below, let $r \in \N$ be greater than one, and set $[r]:=\{0,\ldots ,r-1\}$.
We also let $[r]^{\omega}$ denote the set of all functions from the ordinal $\omega$ to the set $[r]$.
\begin{defn} \label{r-reg1}
We say that $A \subseteq [0,1]$ is \emph{$r$-regular} if there is a B\"uchi automaton $\calA$ with alphabet $\{0, \ldots ,r-1\}$ that recognizes a set $L \subseteq \{0, \ldots ,r-1\}^{\omega}$ such that $(w_i)_{i< \omega} \in L$ if and only if there is $x \in A$ such that 
$$x=\sum_{i=0}^{\infty} \frac{w_i}{r^{i+1}}.$$
Moreover, if this holds we say that $\calA$ \emph{recognizes} $A$.
\end{defn}

\begin{defn}\label{sparse defn}
We say that $A \subseteq [0,1]$ is \emph{$r$-sparse} if it is $r$-regular and the set of length-$n$ prefixes of $r$-representations of elements in $A$ grows at most polynomially in $n$. 
\end{defn}

It is not hard to show that in Definition \ref{sparse defn}, if one automaton recognizing the set $A$ has the property that the set of strings that have an infinite prolongation that is accepted by the automaton grows polynomially, then this holds for all automata recognizing $A$.

Recall that in \cite{vdD85}, van den Dries proves that the expansion of both the real additive group and the reals as an ordered field by a predicate for $2^{\Z}$ has desirable tameness properties, including d-minimality (see \cite{Mi05} for an introduction to the notion) and decidability.
The results of \cite{vdD85} and the following interdefinability result imply many desirable properties hold for the expansion of $(\R,<,+,0,1)$ by a predicate $A$ that picks out an $r$-sparse subset of $[0,1]$.
Throughout this paper we take 
\begin{equation}
 \calR_{A} :=(\R, <,+,0, 1, A) \qquad {\rm and} \qquad   \calR_{r, \ell} :=(\R, <, +, 0,1, r^{-\ell \N}). 
\end{equation}

\begin{theoremA}[Corollary \ref{sparseD}]\label{sparse rN}
Let $r>1$ be a natural number, and suppose $A\subseteq [0,1]^d$ is $r$-sparse.  Then there exists $\ell \in \N$ such that $A$ is $\emptyset$-definable in $\calR_{r, \ell}$,
and the set $r^{-\ell \N}$ is $\emptyset$-definable in $\calR_{A}$.
\end{theoremA}

Recall that $X \subseteq \R^d$ is called a \emph{Cantor set} if it is compact, has no isolated points, and no interior.
From d-minimality, we can conclude that for an $r$-sparse set $A$ the structure $\calR_{A} $ does not define a Cantor set.
Theorem \ref{sparse rN}, in conjunction with the following result, gives us a characterization of what kind of definable sets show up in expansions of the real ordered additive group by a predicate for an $r$-regular subset of $[0,1]$.
Below, for $X \subseteq \R$ let $d_H(X)$ denote the Hausdorff dimension of $X$.

\begin{theoremA}[Theorem \ref{hdimCantor}]\label{tame}
If $A$ is a closed $r$-regular subset of $[0,1]$ such that $0<d_H(A)<1$, then there is a Cantor set definable in $\calR_A$.
\end{theoremA}

In fact, we prove a slightly more technical statement that works in higher arities, i.e.\ $A\subseteq [0,1]^d$; see Theorem Theorem \ref{hdimCantor} for the precise statement
Theorem \ref{tame} demonstrates the connection of fractal dimension, namely Hausdorff dimension, to definability of a more ``pathological'' set, in this case a Cantor set.
What makes a Cantor set ``pathological'' in this setting is that work of Hieronymi and Walsberg \cite{HW19} shows that the expansion of $(\R,<,+,0)$ by a Cantor set is not model-theoretically tame.

Some notions of Shelah-style tameness in the context of model theory include NIP (also known as \emph{not the independence property}) and NTP$_2$ (also known as \emph{not the tree property of the second kind}).
These notions correspond to combinatorial properties of formulas modulo a specific theory, or of the sets definable in the models of a given theory.
The property NIP for formulas corresponds to finiteness of V-C dimension, a notion of interest to some computer scientists.
For definitions of NIP and NTP$_2$ we refer the reader to \cite{S15}.

Let $\pi_i:\R^n \to \R$ denote the projection of $\vec{x} \in \R^n$ onto the $i$-th coordinate.
Connecting the work of Hieronymi and Walsberg with Theorem \ref{tame}, we obtain a theorem that solidifies the notion that for these structures, tameness in the model-theoretic sense and tameness in the sense of fractal geometry completely coincide.

\begin{theoremA}[Theorem \ref{r-reg tame}]\label{r-reg dichotomy}
For $A\subseteq [0,1]^d$ an $r$-regular set such that $d_H(\overline{\pi_i(A)})<1$ for all $i \in \{1,\ldots ,d\}$, the following are equivalent:
\begin{enumerate}
\item{$A$ is $r$-sparse;}
\item{$d_H(\overline{\pi_i(A)})=0$ for all $i\in \{1,\ldots ,d\}$;}
\item{$\calR_A$ is d-minimal;}
\item{The theory of $\calR_A$ has NIP;}
\item{The theory of $\calR_A$ has NTP$_2$.}
\end{enumerate}
\end{theoremA}

Note that in this final theorem, we have dropped the assumption that $A$ is closed, illustrating how well the topological closures of these $r$-regular sets reflect or control the behavior of even the non-closed $r$-regular sets.

\section*{Acknowledgements}
We thank Rahim Moosa and Christa Hawthorne for many valuable contributions to this project and for many helpful comments. 
Many thanks also to Philipp Hieronymi for his insightful thoughts and ideas.
Thanks also to the anonymous referee for their useful comments and suggestions.

\subsection{Background on automata theory.}
Throughout this paper, all terminology, notation, and definitions relating to model theory are drawn from \cite{M02}, unless stated otherwise.
In recent years, there have been new applications of automata theory to the model theory of tame structures.
Uses of finite automata in recent works of model theorists include the introduction of $F$-sets by Moosa and Scanlon in \cite{MS02} and \cite{MS04} with applications to isotrivial Mordell-Lang in positive characteristic, which were shown to be recognized by finite automata in a particular sense by Bell and Moosa in \cite{BM19}.
Other applications include tameness results, namely stability, of the expansion of $(\Z,+)$ by a predicate for a subset of $\Z$ recognized by an automaton with a specific form of alphabet \cite{H20}.

The connections between model theory and automata theory have inspired the work in this paper centered around expansions of the real ordered additive group by sets recognized by B\"uchi automata.
Recent work in this area includes characterizing continuous functions in expansion of the real additive group by sets recognized by B\"uchi automata; in \cite{BG et al.20}, the authors show that a continuous real-valued function on a closed interval whose graph is recognized by a B\"uchi automaton must be locally affine.

The following facts and definitions about finite automata can all be found in \cite{S13}.
A \emph{finite automaton} is a 5-tuple $(Q, \Sigma, \delta, q_0, F)$ such that:
\begin{itemize}
\item{$Q$ is a finite set of states,
}
\item{$\Sigma$ is a finite alphabet,}
\item{$\delta:Q\times\Sigma \to Q$ is a transition function, or rather partial function; $(q_i, \alpha) \mapsto q_j$ or $(q_i, \alpha) \mapsto \emptyset$,
}
\item{$q_0$ is the initial state,}
\item{$F\subseteq Q$ is the set of accept states (also called final states).
}
\end{itemize}
If $\calA$ is a finite automaton and $w = \sigma_1\cdots \sigma_n$ is a finite string generated from $\Sigma$, i.e.,
 $\sigma_i \in \Sigma$ for each $i\leq n$, then a \emph{run} of $\calA$ on $w$ is a sequence of states $q_0, \ldots , q_{n}$ where each $q_i$ is a state in $\calA$, the state $q_0$ is the initial state, and for each $1\leq i\leq n$ we have $\delta(q_{i-1},\sigma_{i}) = q_{i}$.
We say that a finite automaton $\calA$ \emph{accepts} a finite string $w$ generated from $\Sigma$ precisely if a run of $\calA$ on $w$ terminates in an accept state, i.e. a state in $F$.

The definition above is of a \emph{deterministic} finite automaton.
We can generalize this definition to \emph{nondeterministic} finite automata by changing $\delta$ from a partial function $\delta:Q \times \Sigma \to Q$ to a full function on $Q \times \Sigma$ that takes values in $\calP(Q)$, the powerset of $Q$.
In other words, we allow $\delta: Q\times \Sigma \to \calP (Q)$ for the transition function.
The notion of a ``run'' is adapted accordingly.
Conveniently, for finite automata the class of deterministic finite automata and the class of nondeterministic finite automata are equivalent, in the sense that for each non-deterministic finite automaton there is a deterministic one that accepts precisely the same collection of words, and vice versa.

Given a finite alphabet $\Sigma$, we let $\Sigma^*$ denote all finite length strings over $\Sigma$.
Given a set $X$ of finite length strings, we let $X^*= \{x_1x_2\ldots x_n: x_i \in X, n\in \N\}$, where $x_ix_j$ denotes the concatenation of $x_i$ by $x_j$ on right.
We call $X^*$ the \emph{Kleene star} of $X$.
We call a subset $L \subseteq \Sigma^*$ a \emph{language} and we say that automaton $\calA$ \emph{recognizes} $L$ if for all $w \in \Sigma^*$, $\calA$ accepts $w$ if and only if $w \in L$.
Subsets of $\Sigma^*$ recognized by some finite automaton are then called \emph{regular languages}, and they are closed under the process of taking complements, taking finite unions and intersections, concatenation, and the Kleene star operation.

We will often conflate regular languages with what are called \emph{regular expressions}.
Regular expressions of a finite alphabet $\Sigma$ are a class of strings generated from $\Sigma$ via union, concatenation, and Kleene star.
In other words, the regular expressions on $\Sigma$ include the empty set $\emptyset$, the empty string $\varepsilon$, and each character $\sigma \in \Sigma$, as well as all expressions one can generate by starting with those symbols and iteratively applying the operations union, concatenation, and Kleene star.
For more details on finite-state automata, regular languages, and regular expressions, we refer the reader to \cite{AS}.

B\"uchi automata are still given by 5-tuples of the form $(Q, \Sigma, \delta, q_0, F)$ as defined above, the difference is that their inputs are strings from $\Sigma^{\omega} = \{ (\sigma_i)_{i \in \omega} : \sigma_i \in \Sigma\}$ rather than $\Sigma^*$, and the acceptance condition differs.
We define a run of a B\"uchi automaton $\calA$ on an element $w \in \Sigma^{\omega}$ in an analogous manner; a \emph{run} of $\calA$ on $w$ is a sequence of states $(q_i)_{i<\alpha}$ where $\alpha \in \mathbb{N}$ or $\alpha = \N$, and each $q_i$ is a state in $\calA$, the state $q_0$ is the initial state, and for each $i < \alpha$ we have $\delta(q_i,\sigma_i) = q_{i+1}$.
We say that a B\"uchi automaton $\calA$ accepts a string $w \in \Sigma^{\omega}$ precisely if a run of $\calA$ on $w$ includes an accept state infinitely often.
Note that in the definition of a run for B\"uchi automata, if $\alpha \neq \N$, then the string $w$ cannot be accepted, since the run terminates despite $w$ being an infinite string.

Since for B\"uchi automata the alphabet $\Sigma$ is still a finite set, the languages that B\"uchi automata recognize can all be viewed as subsets of a Cantor space, namely the space $\{0,\ldots ,n-1\}^{\omega}$ if $|\Sigma|=n$.
Now we can restate Definition \ref{r-reg1} more clearly, and for subsets of $[0,1]^m$.
Set $[r] := \{0, \ldots ,r-1\}$, so that $([r]^m)^{\omega}=\{\{0, \ldots ,r-1\}^m\}^{\omega}$.
For $(x_1, \ldots ,x_m) \subseteq [0,1]^m$, we say that $(w_{1,i},\ldots,w_{m,i})_{i<\N} \in ([r]^m)^{\omega}$ is an \emph{$r$-representation} of $(x_1, \ldots ,x_m)$ if for each $j \leq m$ the following holds: 
$$x_j = \sum_{i=0}^{\infty}\frac{w_{j,i}}{r^{i+1}}.$$

\begin{defn}
\label{r-regm}
For $m,r \in \N_{>0}$ with $r>1$, we say that $X \subseteq [0,1]^m$ is \emph{$r$-regular} if there is a B\"uchi automaton $\calA$ such that the following holds:
\begin{enumerate}
\item{for each $w \in ([r]^m)^{\omega}$, if $w$ is not an $r$-representation for any $x \in X$, then $\calA$ does not accept $w$;}
\item{for each $x \in X$, there exists a $w \in ([r]^m)^{\omega}$ such that $w$ is an $r$-representation of $x$ and $\calA$ accepts $w$.}
\end{enumerate}
If such an $\calA$ exists, we say that $\calA$ \emph{recognizes} $X$.
\end{defn}
Note that unlike with finite automata, reversing the accept and non-accept states of a B\"uchi automaton does not yield an automaton that recognizes the complement of the language that the original automaton recognized.
Just as there is a correspondence between finite automata and regular languages, there is an analogous correspondence between B\"uchi automata and regular $\omega$-languages.
We call a subset $L$ of $\Sigma^{\omega}$ an $\omega$-\emph{language}.  

We say that the B\"uchi automaton $\calA$ \emph{recognizes} $L$ if for all $w \in \Sigma^{\omega}$, $\calA$ accepts $w$ if and only if $w \in L$.
We say that $L \subseteq \Sigma^{\omega}$ is a \emph{regular $\omega$-language} if there is a B\"uchi automaton that recognizes $L$.
It is again true, though less easy to show, that regular $\omega$-languages are closed under the boolean operations union, intersection, and complementation.
Suppose that $V \subseteq \Sigma^*$ is a regular language and $L \subseteq \Sigma^{\omega}$ is a regular $\omega$-language.
Then the following are regular $\omega$-languages as well:
\begin{itemize}
\item{$VL = \{vw \in \Sigma^{\omega}: v \in V, w \in L\}$,}
\item{$V^{\omega} = \{v_1v_2v_3\ldots \in \Sigma^{\omega}: v_i \in V \}$.}
\end{itemize}

There is an important theorem of B\"uchi's that makes working with regular $\omega$-languages significantly easier.
We will use the following theorem frequently in this paper without direct reference.

\begin{thm}[\cite{B62}]\label{buchi}
For every $L \subseteq \Sigma^{\omega}$ recognized by a B\"uchi automaton, there are regular languages $V_1, \ldots ,V_k,W_1, \ldots ,W_k \in \Sigma^*$ such that 
$$L=\bigcup_{i=1}^k V_iW_i^{\omega}.$$
\label{thm:buchi}
\end{thm}
\noindent 
We note that the above decomposition of $L$ given in the statement of Theorem \ref{thm:buchi} is not unique.
\begin{defn}
Let $\Sigma$ be a finite alphabet and let $L$ be a subset of $\Sigma^{\omega}$ recognized by a B\"uchi automaton.  Given a decomposition of $L$ as a finite union of the form $VW^{\omega}$ as in Theorem \ref{thm:buchi}, we will call $VW^{\omega}$ a \emph{V-W component} of $L$.
\end{defn}

\begin{defn} 
Suppose that $\calA$ is a B\"uchi automaton.
\begin{enumerate}
\item{For states $p$ and $q$ in an automaton, we say $q$ is \emph{accessible} from $p$ if there exists a run of some non-trivial word from $p$ to $q$.}
\item{We say that $\calA$ is \emph{weak} if for any states $p$ and $q$ in $\calA$, whenever $p$ and $q$ are accessible to and from each other, then one is an accept state precisely if the other is an accept state.}

\item{We say that $\calA$ is \emph{trim} if each state is accessible from the start state and some accept state is accessible from each state (via a nontrivial path).}
\item{If $\calA$ is trim, we say that $\calA$ is \emph{closed} if every state is an accept state.}
\item{Let $\overline{\mathcal{A}}$ be the automaton resulting from making every state of $\mathcal{A}$ accepting.}
\end{enumerate}
\end{defn}

\begin{fact}[{\cite[Remark 59]{CLR15}}]
If the set $X\subseteq [0,1]$ is recognized by a B\"uchi automaton $\mathcal{A}$, then its closure in the order topology $\overline{X}$ is recognized by $\overline{\mathcal{A}}$.
\end{fact}

Some of the earliest connections made between B\"uchi automata and model theory come from the work of Boigelot, Rassart, and Wolper in \cite{BRW98}.
In that extended abstract, the authors introduce the following ternary predicate.
\begin{defn}
Let $V_r(x,u,k)$ be a ternary predicate on $\mathbb{R}$ that holds precisely if $u=r^{-n}$ for some $n \in \mathbb{N}_{>0}$ and the $n$-th digit of a base-$r$ representation of $x$ is $k$.
\end{defn}
\noindent By examining the expansion of the real ordered additive group by this predicate, they prove the following definability result.
\begin{thm}[\cite{BRW98}]
A subset $X\subseteq [0,1]^n$ is $r$-regular if and only if $X$ is $\emptyset$-definable in $(\mathbb{R},<,0,+, V_r)$.
\end{thm}
\noindent This theorem has given rise to many further applications of automata theory to model theory.

\subsection{Background on fractal dimensions.}
Finally, we will recall the definitions of Hausdorff measure and Hausdorff dimension for subsets of $\R$.
For $X \subseteq \R$ and $\varepsilon >0$, let $\calB_{\varepsilon}$ be the collection of all countable covers of $X$ by closed intervals of diameter at most $\varepsilon$.
We let $\Diam(U) \in \R \cup \{ \infty \}$ denote the diameter of $U \subseteq \R$.
Given a nonnegative real number $s$, we define the Hausdorff $s$-measure of $X \subseteq \R$ as follows: 
$$\mu^s_H(X) := \lim_{\varepsilon \to 0} \inf_{B \in \calB_{\varepsilon}}  \left\{ \sum_{U\in B} (\Diam U)^s  \right\}.$$
The \emph{Hausdorff dimension} of $X$ is written
$d_H(X)$, and is the unique $s \in \R_{\geq0}$ such that if $s'>s$, then $\mu^{s'}_H(X) = 0$ and if $s'<s$, then $\mu^{s'}_H(X) = \infty$.
Note that it is quite possible that $\mu^{s}_H(X) = 0$, and it is possible that $s=0$ (indeed, if $X$ is a single point, then $d_H(X)=0$).
If $X \subseteq \R^n$, the Lebesgue measure of $X$ is positive and finite if and only if $\mu_H^n(X)$ is positive and finite. 
Thus it is not possible that $\mu^{s}_H(X) = \infty$ for all $s>0$.

Hausdorff dimension behaves well with respect to subsets and unions, as shown by this fact:
\begin{fact}[\cite{F03}]
Let $X$, $Y$, and $(Y_i)_{i<\N}$ be subsets of $\R$.
    \begin{enumerate}
        \item If $X \subseteq Y$, then $d_H(X) \leq d_H(Y)$.
        \item If $X = Y_1 \cup Y_2$, then $d_H(X) = \max(d_H(Y_1), d_H(Y_2))$.
        \item If $X = \bigcup_{i<\N} Y_i$, then $d_H(X) = \max_{i< \N}\{d_H(Y_i)\}$.
    \end{enumerate}
\end{fact}

\subsection{Background on S-unit theory.}

In this brief section we give an overview of the theory of $S$-unit equations. Specifically, we require a quantitative version of a result due to Evertse, Schlickewei and Schmidt (see~\cite[Theorem~1.1]{ESS02} and also~\cite[Theorem~6.1.3]{EG}).  

\begin{definition} Let $L$ be a field and let  $a_1,\dots,a_m\in L$ be nonzero elements of $L$. Then a solution $x_1,\dots,x_m\in L$ to the equation
$$
a_1 x_1+\dots + a_m x_m = 1
$$
is called \emph{nondegenerate} if the left side has no non-trivial vanishing subsums; i.e., if $\displaystyle \sum_{i \in I} a_i x_i \neq 0$ for each non-empty subset $I$ of $\{ 1, \dots, m \}$.  In general, we will say a sum $\sum_{i=1}^m x_i$ is \emph{nondegenerate} if it has no non-trivial vanishing subsums.
\end{definition}

The following theorem is usually referred as the $S$-unit theorem.

\begin{thm}\label{thm:unit}
  Let $L$ be a field of characteristic zero, let
  $a_1,\dots,a_m$ be nonzero elements of $L$, and let $H\subset(L^\times)^m$ be a finitely generated  multiplicative group.
  Then there are only finitely many nondegenerate solutions $(y_1,\dots,y_m)\in H$ to the equation
  $a_1y_1+\dots+a_my_m=1$.
\end{thm}

(See Theorem~6.1.3 in~\cite{EG}---although this theorem is only stated for $m\ge 2$, the case $m=1$ is immediate.)


To show that sparse $k$- and $\ell$-regular sets have finite intersection when $k$ and $\ell$ are multiplicatively independent, we will also use the following theorem of Evertse, Schlickewei, and Schmidt (2002) (see \cite[Theorem 6.1.3]{EG}).

\begin{thm} Let $K$ be a field of characteristic zero, let $n\ge 2, r\ge 1$, let $\Gamma\le K^*$ be a finitely generated subgroup of rank $r$, and let $a_1,\ldots ,a_n$ be nonzero elements of $K$.  Then the number of nondegenerate solutions to $$a_1x_1+\cdots +a_n x_n= 1$$ with $(x_1,\ldots ,x_n)\in\Gamma^n$ is at most $\exp((6n)^{3n}(r+1))$.
\label{thm:eg}
\end{thm}

\section{Preliminaries on sparsity}

Let $r \in \N_{>1}$, and set $[r]:=\{0,1, \ldots ,r-1\}$ for the remainder of this paper.
Note that we will often switch between elements of $[0,1]$ and their $r$-representations, and may sometimes say that an automaton $\calA$ accepts \emph{the} $r$-representation of $x \in [0,1]$.
For the countable subset of $[0,1]$ whose elements have multiple (i.e. exactly two) $r$-representations, we mean that $\calA$ accepts \emph{at least one} of the $r$-representations of $x$.
For ease of switching between $x$ and its $r$-representation, we will define a valuation for elements of $[r]^{\omega}$.

\begin{defn}
Define $\nu_r:[r]^{\omega} \to [0,1]$ by:
$$\nu_r(w) = \sum_{i=0}^{\infty} \frac{w_i}{r^{i+1}}$$
where $w=w_0w_1w_2\ldots$ with $w_i \in [r]$ for each $i \in \N$.
\end{defn}

Note that the equivalence relation $v \equiv w \iff \nu_r(v) = \nu_r(w)$ induces equivalence classes of size at most two.
As noted above, only countably many elements in $[r]^{\omega}$ are not the unique element of their $\nu_r$-equivalence class.
For $L \subseteq [r]^{\omega}$ an $\omega$-language, we set $\nu_r(L):=\{ \nu_r(w):w \in L\}$.
If $L \subseteq ([r]^m)^{\omega}$, set
$$\nu_r(L):=\{ (\nu_r(w_1), \ldots ,\nu_r(w_m)): w_1, \ldots,w_m \in [r]^{\omega}, (w_{1,i},\ldots ,w_{m,i})_{i<\omega} \in L\}.$$

We can now expand Definition \ref{sparse defn} to languages in $([r]^m)^{\omega}$.

\begin{defn} We say that $L \subseteq ([r]^m)^{\omega}$ is \emph{sparse} if it is regular (as an $\omega$-language) and the set of length-$n$ prefixes of the elements of $L$ grows at most polynomially in $n$. 

We say that $X\subseteq [0,1]^m$ is \emph{$r$-sparse} if $X = \nu_r(L)$ for some sparse $L \subseteq ([r]^m)^{\omega}$, and we say that a B\"uchi automaton is \emph{sparse} if the language that it recognizes is sparse.
\end{defn}

There is a dichotomy for the set of length-$n$ prefixes of an $r$-regular $\omega$-language: either the set of length-$n$ prefixes grows polynomially in $n$, or it grows exponentially in $n$ (see Remark \ref{growthrmk}).

Suppose $L \subseteq ([r]^m)^{\omega}$ is $r$-regular.
We now use a characterization of sparseness from \cite{BM19} to show that if $L$ is sparse, then $L$ is a finite union of sets of the following form:
\begin{equation}\label{sparse form}
u_1v_1^* u_2v_2^*\ldots u_{d-1}v_{d-1}^* u_d v_d^{\omega}
\end{equation}
where $u_i,v_i \in ([r]^m)^{*}$ for all $i\leq d$, and $v_1, \ldots ,v_d$ are non-empty strings.


\begin{prop}\label{sparseequiv}
A language $L \subseteq ([r]^m)^{\omega}$ is sparse if and only if $L$ is a finite union of sets of the form (\ref{sparse form}).
\end{prop}
\begin{proof}
First suppose that $L$ is sparse. By Theorem \ref{buchi}, $L$ is a finite union of sets of the form $VW^{\omega}$ where $V,W \subseteq ([r]^m)^*$ are regular, and so it suffices to show that each V-W component of $L$ has the form given in 
(\ref{sparse form}). 

We suppose towards a contradiction that this is not the case for some V-W component $VW^{\omega}$. Then using the characterization of sparseness from Proposition 7.1 in \cite{BM19}, we see that either $V$ contains a sublanguage of the form $u\{a,b\}^*v$ with $a$ and $b$ distinct words of the same length, or $W^*$ contains at least two distinct words of the same length. In both cases we see that the number of distinct length-$n$ prefixes of $VW^{\omega}$ grows exponentially with $n$, contradiction the fact that $L$ is sparse. It follows that $V$ is sparse and by Corollary 6.2.5 in \cite{L02} that $W^*$ is a regular sublanguage of $c^*$ for some non-trivial word $c$.  Thus $W^{\omega}=c^{\omega}$ and by the characterization of sparse languages in 
Proposition 7.1 in \cite{BM19} we see that $VW^{\omega}$ has the desired form.

Conversely if $L$ is a finite union of sets of the form (\ref{sparse form}), then the length $n$ prefixes are the same as those of $u_1v_1^* u_2v_2^*\ldots v_{d-1}^* u_d v_d^*$, which by \cite{BM19} grow polynomially in $n$, as desired.

\end{proof}

Assuming the automata for our regular sets are trim, the closure of the automaton accepts the closure of the regular set. 
\begin{rmk} \label{rmk:samelength} In Proposition \ref{sparseequiv}, we may further refine and assume that $|v_1|=|v_2|=\cdots =|v_d|$ in each set our union comprises. The reason for this is that if we let 
$N$ denote the least common multiple of $|v_1|,\ldots ,|v_d|$, then we may replace each $v_i$ by $v_i^{N/|v_i|}$ and replace each $u_i$ by $u_i v_i^{a_i}$ where $a_1,\ldots ,a_d$ vary over elements of the set $$\{0,\ldots , N/|v_1|-1\}\times \cdots \times \{0,\ldots ,N/|v_d|-1\}.$$
\end{rmk}
\begin{rmk}\label{sparse-closure}
For $r$-sparse subsets of $[0,1]^m$, the closures of a set of the form $\nu_r(u_1v_1^* u_2v_2^*\ldots v_{d-1}^* u_d v_d^{\omega})$ with $v_1,\ldots ,v_d$ non-empty words is given by the finite union $$\nu_r(u_1v_1^*\ldots u_dv_d^{\omega})\cup \nu_r(u_1v_1^*\ldots u_{d-1}v_{d-1}^{\omega})\cup\cdots\cup \nu_r(u_1v_1^{\omega}).$$
In particular, the closure of a sparse regular set is again a sparse regular set.
\end{rmk}

\begin{rmk} \label{sparse-description0} If one does a computation analogous to the one done in \cite[Remark 3.4]{AB1}, we see that a set of the form 
$\nu_r(u_1v_1^* u_2v_2^*\ldots v_{d-1}^* u_d v_d^{\omega})$ with $v_1,\ldots ,v_d$ non-empty words has the form
\begin{equation} \label{eq: form0}
\left\{ c_0  + c_1 r^{-\delta_1 n_1} + c_2 r^{-\delta_1 n_1 - \delta_{2} n_{2}}\cdots + c_{d-1} r^{-\delta_1 n_1-\cdots -\delta_{d-1} n_{d-1}} \colon n_1,\ldots ,n_{d-1}\ge 0\right\},\end{equation}
where $c_0,\ldots ,c_{d-1}$ are elements of $\mathbb{Q}^m$ and $\delta_1,\ldots ,\delta_{d-1}$ are positive integers with $\delta_i=|v_i|$ for $i=1,\ldots ,d-1$.
\end{rmk}

For the remainder of this section, we restrict our attention to the set of closed $r$-regular subsets of $[0,1]^m$.
The characterization of non-sparse regular languages from \cite{BM19} implies that if $X\subseteq [0,1]^m$ is closed but not $r$-sparse, then $X$ contains a set of the form $Y = \nu_r(u \{a,b\}^{\omega})$, i.e., a translate of a Cantor set.
It is not clear \textit{a priori} that $Y$ is definable in the language of ordered additive groups expanded by a predicate for $X$.

The following lemma will be crucial for a subsequent characterization of expansions of the real ordered additive group by a predicate for an $r$-regular set, among other things.

\begin{lem}\label{TFAE}
Suppose that $X$ is a closed r-regular subset of $[0,1]^m$.
The following are equivalent:
\begin{enumerate}
\item $X$ is $r$-sparse;
\item $X\subseteq \Q^m$;
\item $X$ is countable;
\item the set of accumulation points of $X$ is countable.
\end{enumerate}
\begin{proof}
Let $L\subseteq ([r]^m)^{\omega}$ be such that $\nu_r(L) = X$.
For (1)$\implies$(2), we observe that since $X$ is $r$-sparse, by Proposition \ref{sparseequiv} it is a finite union of sets of the form $\nu_r(u_1v_1^*\ldots u_{n}v_{n}^{\omega})$. 
Hence each element of $L$ is tail-equivalent to one of finitely many sequences of the form $v_n^{\omega}$.
Recall that an $r$-representation of a real number is eventually periodic if and only if that number is rational.
Since $v_n^{\omega}$ is by construction a periodic $r$-representation, anything element of $([r]^m)^{\omega}$ that is tail-equivalent to it is eventually periodic as well, and thus in $\Q^m$.

It is obvious that (2)$\implies$(3).
For (3)$\implies$(4), since $X$ is closed, it contains all of its accumulation points. 
So as a subset of a countable set the accumulation points of $X$ are countable as well.
Finally, for (4)$\implies$(1), if the set of accumulation points of $X$ is countable, then there is no subset of $L$ of the form $\nu_r (u\{a,b\}^*wv^{\omega})$ with $a,b,u,v,w \in ([r]^m)^*$ and $|a|=|b|$ but $a \neq b$.
If there were, the set $\nu_r(u \{a,b\}^{\omega})$ would be a subset of the set of accumulation points of $X$.
So we deduce by Proposition \ref{sparseequiv} that $X$ is $r$-sparse.
\end{proof}
\end{lem}

\begin{rmk}\label{growthrmk}
As a corollary of Lemma \ref{TFAE}, it is immediate that for any $r$-regular $\omega$-language $L \subseteq ([r]^m)^{\omega}$, either the size of the set of length-$n$ prefixes of words in $L$ eventually is bounded by a polynomial in $n$, or eventually grows exponentially.
\end{rmk}

\begin{fact}[\cite{CHM02}]\label{syndetic} If $X \subseteq [0,1]$ is an uncountable $r$-regular set, then there exists $d \in \N$ such that $d\cdot X:=X+X+\cdots+X$ contains a closed interval.
\end{fact}

The above fact means that there is another characterization of sparseness: a closed $r$-regular subset $X$ of $[0,1]$ is sparse if and only if for all $d\ge 1$ the following set
$$d\cdot X :=\{x_1+\cdots +x_d \colon x_1,\ldots ,x_d\in X\}$$ 
does not have interior.

Finally, we establish the connection between Hausdorff dimension and closed, $r$-sparse subsets of $[0,1]^m$.

\begin{lem}\label{dimH0}
If $X$ is a closed regular subset of $[0,1]^m$, then $X$ is $r$-sparse if and only if $d_H(\pi_i(X))=0$ for all $i \leq m$.
\end{lem}
\begin{proof}
Suppose that $X \subseteq [0,1]^m$ is recognized by the automaton $\calA$.
First suppose that the set of words that $\calA$ accepts is not sparse.  In particular, it includes a subset of the form $u\{a,b\}^*wv^{\omega}$ with $a \neq b$ and $|a|=|b|$.
Then the automaton $\calA$ must also accept all words in $u\{a,b\}^{\omega}$, due to the fact that in a closed automaton all states with a path to itself are accept states. 
There exists some $i\leq m$ such that the words in coordinate $i$ are elements of a set of the form $u_i\{a_i, b_i\}^{\omega}$ with $a_i\neq b_i$.

Let $X_i\subseteq \R$ denote the set recognized by the ``projection'' of the automaton onto the $i$-th coordinate, which we will call $\calA_i$. Then $X_i$ contains the set 
$Y:=\nu_r(u_i\{a_i, b_i\}^{\omega})$, and so the Hausdorff dimension of $X_i$ is at least as large as that of $Y$. Hence it suffices to show that $Y$ has strictly positive Hausdorff dimension.  This follows immediately from \cite[Theorem A]{BlockGormanSchulz}.

Next suppose that $X$ is sparse.
Then $X \subseteq \Q^m$ by Lemma \ref{TFAE}.
Since Hausdorff dimension is an increasing function with respect to the subset relation, and Hausdorff dimension of projections of $\Q^m$ are all zero, we are done.
\end{proof}
\section{$r$-sparse sets and d-minimality}

The first step toward proving Theorem \ref{sparse rN} is establishing the $\emptyset$-definability of each $r$-sparse subset $X$ of $[0,1]^m$ in $\calR_{r,\ell} =(\R, <, +, 0, 1, r^{-\ell \N})$ for some $\ell \in \N$.
In fact, we can take this $\ell$ to be the least common multiple of $\{|u_{1}|,|v_{1}|, \ldots ,|u_{d}|,|v_{d}| \}$, where the words $u_1,v_1, \ldots ,u_d,v_d$ witness that $X$ is a finite union of languages of the form (\ref{sparse form}).


\begin{rmk}\label{rmk:33}
If $k$ is a divisor of $\ell$, then $r^{-k\N}$ is definable in $r^{-\ell \N}$.  In particular, $r^{-\N}$ is definable in $r^{-\ell \N}$.
\end{rmk}

\begin{prop}\label{sparseDmin}
Every $r$-sparse set $A \subseteq [0,1]^m$ is $\emptyset$-definable in $\calR_{r, \ell}$ for some $\ell \in \N$.\end{prop}
\begin{proof}
By Remarks \ref{rmk:samelength} and \ref{sparse-description0}, $A$ is a finite union of sets of the form
\begin{equation}
    \label{eq:AA}
\left\{ c_0  + c_1 r^{-\delta n_1} + c_2 r^{-\delta n_1 - \delta n_{2}}\cdots + c_{d-1} r^{-\delta n_1-\cdots -\delta n_{d-1}} \colon n_1,\ldots ,n_{d-1}\ge 0\right\},\end{equation} where $d\ge 1$, $\delta$ is a positive integer and $c_0,\ldots ,c_{d-1}$ are rational vectors with $c_1,\ldots ,c_{d-1}$ nonzero.

Observe that by Remark \ref{rmk:33}, if $A_1,\ldots ,A_s$ are sets that are respectively definable in $\calR_{r, \ell_1},\ldots \calR_{r, \ell_s}$, then they are all definable in $\calR_{r, \ell_1\cdots \ell_s}$ and so it suffices to show the result when $A$ is a set of the form given in (\ref{eq:AA}). 

Notice that if we take $\ell=\delta$ then we can define the above set by the rule
$$\exists x_1, x_2,\ldots ,x_{d-1}\in r^{-\delta \mathbb{N}}(z= c_0 + c_1 x_1+\cdots + c_{d-1} x_{d-1}) \wedge (x_1\le x_2\le \cdots \le x_{d-1}).$$

\end{proof}

We note that the above result is optimal in the sense that if $A$ is an $r$-sparse set definable in $\calR_{r,\ell}$ with $\ell>1$, in general $A$ will not also be definable in $R_{r,m}$ for some $1\leq m < \ell$.
To see this, let $T_r:=Th(\R,<,+,0,r^{-\N})$.
Consider $\tilde{\calR}:=(R,<,+,0, D) \models T_r$  where $\R \subseteq R$ and $D$ is the interpretation of $r^{\N}$ in $\tilde{\calR}$, and define $\lambda:R \to R$ as follows:
$$t \mapsto \begin{cases}
1, & t \geq 1 \\
0, & t \leq 0 \\
\min ([t,1) \cap D), & t \in (0,1).\\
\end{cases} $$
Effectively, the function $\lambda$ maps any element $t$ of $(0,1)$ to the smallest element of $D$ greater than $t$, and  all other positive $t$ get mapped to $1$, while all non-positive $t$ are sent to $0$.
Note that all models of $T_r$ are interdefinable with a model of $T_{\lambda}:=Th(\R,<,0,+, \lambda)$ that has the same underlying set, and vice versa.
Observe that $T_{\lambda}$ is axiomized by the axioms of ordered divisible abelian groups plus the following axioms, which are adapted from \cite{MT06}:
\begin{enumerate}[(i)]
\item $s\leq t \rightarrow \lambda (s) \leq \lambda (t)$,
\item $t \geq 1 \rightarrow \lambda(t)=1$,
\item $t \leq 0 \rightarrow \lambda (t)=0$,
\item $t \in (0,1) \rightarrow \frac{1}{r}\lambda(t) < t \leq \lambda (t)$,
\item $\frac{1}{r}\lambda(t)<s \leq \lambda(t) \rightarrow \lambda(s) = \lambda(t)$,
\item $t \in [0,1] \rightarrow (\lambda(t)=t \leftrightarrow \lambda (\frac{t}{r})=\frac{t}{r})$.
\end{enumerate}

Note that these axioms are a bit redundant, but it will be helpful to have them written this way.
Below, let $x\ll1$ denote the statement $\forall r \in \R (0<x<r)$, i.e. the property that $x$ is infinitesimal with respect to the real numbers.
For a subset $X$ of an expansion of $\R$, let $X^{\ll 1}$ denote the set of infinitesimal elements in $X$ (with respect to $\R$).

Recall that for a set $A \subseteq \R$ and a language $\calL$, we define $\dcl_{\calL} (A)$ to be the set of all $r \in \R$ such that the singleton set $\{r\}$ is $\calL$-definable with parameters in $A$.

\begin{lem}\label{lambdaQE}
Suppose that $\calS:=(S, <,+,0,\lambda)$ is a $|\R|$-saturated elementary extension of the structure $\calR_r=(\R,<,+,0, \lambda) \models T_{\lambda}$. 
Let $\calL=\{<,+,0\}$, and let $\calL_{\lambda}=\{<,+,0,\lambda \}$ where the interpretation of $\lambda(\R)$ in $\calR_r$ is $r^{-\N}\cup \{0\}$.
Suppose $b$ and $b'$ are in $\lambda(S)$ and satisfy the same $\calL$-type over $\R$.  Then $b$ and $b'$ satisfy the same $\calL_{\lambda}$-type over $\R$.
\begin{proof}

To prove the claim, it suffices to show that if $e:S \to S$ is an $\calL$-elementary map that fixes $\R$ pointwise and maps $b$ to $b'$, then $e$ is also an $\calL_{\lambda}$-elementary map.
Note that any $\calL$-embedding $e:S \to S$ is also an $\calL$-elementary map because the reduct of $\calS$ to $\calL$ has quantifier elimination.
By Blum's criterion for quantifier elimination (Theorem 3.1.4 in \cite{M02}), to show quantifier elimination for $T_{\lambda}$ it suffices to show the following:
if $\calA=(A, <,+,0,\lambda), \calB=(B, <,+,0,\lambda)  \models T_{\lambda}$ and $\calS$ is an $|A|^+$-saturated elementary extension of $\calA$, then for any $b \in B$ there is $b' \in S$ such that there is an $\calL_{\lambda}$ embedding $e:B \to S$ of $\dcl (A \cup \{b\})$ into $\calS$ that maps $A$ to $A$ pointwise, maps $b$ to $b'$, and the image of $e$ in $\calS$ is $\dcl (A \cup \{b'\})$.
Hence, both to prove quantifier elimination for $T_{\lambda}$ and to prove the claim, it suffices to show that for any $b\in B \setminus A$, if $b' \in S$ satisfies the same cut over $A$ as $b$ and $\lambda^{\calB}(b)=b \iff \lambda^{\calS}(b')=b'$, then any $\calL$-embedding $e:\dcl (A \cup \{ b\}) \to S$ that fixes $A$ pointwise and maps $b$ to $b'$ is also an $\calL_{\lambda}$-embedding.

Let $\calB=(B, <,+,0,\lambda)  \models T_{\lambda}$, let $\calA$ be a substructure of $\calB$, and let $(\calS, \lambda^*)$ be an $|A|^+$-saturated elementary extension of $\calA$.
Suppose that $b \in \lambda(B) \setminus \lambda (A)$, and note that by saturation of $\calS$, there is an element $b' \in \lambda^* (S)$ such that $b$ and $b'$ satisfy the same cut over $A$.
Let $e:\dcl (A \cup \{b\}) \to S$ map $A$ to itself pointwise, and let $e(b)=b'$.
We note that $e$ is indeed an $\calL$-embedding over $A$, since the fact that $b$ and $b'$ satisfy the same cut over $\calA$ implies that they satisfy the same $\calL(A)$-formulas.
To demonstrate that $e$ is also an $\calL_{\lambda}$-embedding over $A$, it suffices to show that $\lambda(\dcl_{\calL}(A \cup \{b\})) \subseteq \dcl_{\calL}(A \cup \{b\})$ and $e(\lambda(x)) = \lambda^* (e(x))$ for all $x \in \dcl_{\calL}(A \cup \{b\})$, because from this we conclude $e$ respects all atomic $\calL_{\lambda}(A)$-formulas, and thus all quantifier-free ones by induction.

Given $s \in S$, define $r^{\Z}s:= \{r^zs: z\in \Z\}$, and for any $X \subseteq S$ let $H(X)$ denote the convex hull of $X$ in $S$.
Recall the following fact from \cite{MT06}: If $s \ll 1$ and $A \cap H( r^{\Z}s ) = \emptyset$, then 

\begin{equation}\label{MTlemma}
\{ x \in \dcl_{\calL} (A \cup \{ s \}): x \ll 1\} \subseteq \bigcup_{\{a \in A^{\ll 1}\} } H ( r^{\Z}\lambda (a)) \cup H(r^{\Z}s).
\end{equation}
Suppose that $c \in C:= \dcl_{\calL}(A \cup \{b\})$.  
If $c \in \dcl_{\calL}(A)$, then it is immediate that $\lambda(c) \in A$ since $\calA$ is a substructure of $\calB$.
So suppose that $c \in C\setminus \dcl_{\calL}(A)$, and suppose also that $0< c \ll 1$, since otherwise it is clear by definition of $\lambda$ that $\lambda(c) \in  \dcl_{\calL} (A \cup \{b \})$.
By axiom (iv), we observe that $c<\lambda(c)<rc$, hence $\lambda(c) \in H(r^{\Z}c)$.
Since we assume $c \in \{ x \in \dcl (A \cup \{ s \}): x \ll 1\}$, we know that either $c \in H ( r^{\Z}\lambda (a))$ for some $a \in A$, in which case there is $z \in \Z$ for which $r^z\lambda(a)\geq c > r^{z+1}\lambda(a)$ so by axiom (v) we conclude $\lambda(c)=\lambda(r^za)$. 
Otherwise, we know $c \in H(r^{\Z}b)$, so by a similar argument, and the fact that $b=\lambda(b)$, we conclude $\lambda(c) = r^zb$ for some $z \in \Z$.
In either case, we conclude that $\lambda(c) \in C$, as desired.

Finally, suppose that $x \in C$, and for contradiction we suppose that $e(\lambda(x)) \neq \lambda^*(e(x))$.
Without loss of generality, we may assume that $e(\lambda(x))<\lambda^*(e(x))$.
Since $e(a)=a$ for all $a \in \calA$, we know $\lambda(a) = \lambda^*(a)$ for all $a \in \calA$.
Note also that $e$ preserves cuts over $A \cup \{b\}$, so the cut of $e(x)$ over $A \cup \{b^*\}$ is the image under $e$ of the cut of $x$ over $A\cup \{b\}$.
If $x>q$ for some $q \in \Q^{\calS}$, then it must be the case that $\lambda^*(e(x))=e(\lambda(x))$, since $e(q^{\calA}) = q^{\calS}$ for all $q \in \Q$, and for each $q \in \Q$ we know that $\lambda(q)$ is the least power of $\frac{1}{r}$ greater than $q$.
So we conclude that $x\ll 1$.
By (\ref{MTlemma}), we obtain the following:
$$x \in \bigcup_{1 \gg a \in A} H(r^{\Z}\lambda(a)) \cup H(r^{\Z}b).$$
Since $e$ is $\calL$-elementary, we conclude the following:
$$e(x) \in \bigcup_{1 \gg a \in A} H(r^{\Z}\lambda(a)) \cup H(r^{\Z}b^*).$$

If $x \in H(r^{\Z}\lambda(a))$, then by definition, there exist $z,z' \in \Z$ such that $r^{z'} \lambda(a)\leq x \leq r^{z} \lambda(a)$ for some $a \in A$.
Hence there is some $k\in \N$ with $z < k \leq z'$ such that $r^{k-1} \lambda(a) \leq x \leq r^k \lambda(a)$.
Since $r^{k-1} \lambda(a), r^k \lambda(a) \in \lambda (\calA)$, we conclude that $\frac{1}{r}\lambda(r^{k} \lambda(a)) \leq x \leq \lambda(r^k \lambda(a))$.
So by axiom (v), we know $\lambda^*(x)=\lambda(r^k\lambda(a))$.
As $e$ is a $\calL$-elementary map preserving $A$, we conclude that $\frac{1}{r}\lambda(r^{k} \lambda(a)) \leq e(x) \leq \lambda(r^k \lambda(a))$.
Hence we know $\lambda^*(e(x))=\lambda(r^k\lambda(a))=\lambda^*(x)$.

Otherwise, we must conclude $x \in H(r^{\Z}b)$ and $e(x) \in H(r^{\Z}b^*)$.
In this case we similarly conclude that there is some $k \in \Z$ such that $r^{k-1}b<x \leq r^kb$, and that there is some $k' \in \Z$ such that $r^{k'-1}b^*<e(x)\leq r^{k}b^*$.
Since $\lambda(b)=b$ and $\lambda(b^*)=b^*$, we see that $r^{k-1}\lambda(b)<x \leq r^k\lambda(b)$
 and $r^{k-1}\lambda(b^*)<x \leq r^k\lambda(b^*)$.
Hence  $\frac{1}{r}\lambda(r^{k}b)<x \leq \lambda(r^kb)$
 and $\frac{1}{r}\lambda(r^{k'}b^*)<e(x) \leq r^k\lambda(r^{k'}b^*)$, by repeated application of axiom (vi).
So by axioms (v) and (vi), we see $\lambda(x) =r^kb$ and $\lambda^*(e(x))=r^{k'}b^*$.
If $e(\lambda(x))<\lambda^*(e(x))$, then $e(r^{k-1}b)=r^{k-1}b^*<e(x)\leq e(\lambda(x))=e(r^kb)=r^kb^*<\lambda^*(e(x))$.
As $r^{k'-1}b^*<e(x)\leq \lambda^*(e(x))=r^{k'}b^*$, we conclude that $r^kb^*<\lambda^*(e(x))=r^{k'}b^*$.
Yet this yields both $e(x)\leq r^{k}b^*<r^{k'}b^*$ and also $r^kb^*\leq r^{k'-1}b^*<e(x)$, a contradiction.
So it must be the case that $e(\lambda(x)) = \lambda^*(e(x))$, and we conclude that $e$ is in fact an $\calL_{\lambda}$-embedding over $A$.
\end{proof}
\end{lem}

\begin{fact}
For $m>1$ the set $r^{-m\N}$ is not definable in the structure $(\R, <, +,r^{-\N})$.
\begin{proof}
Assume for sake of contradiction that the set $r^{-m\N}$ is defined by the $\calL_{\lambda}$-formula $\phi(x)$, with parameters in $\R$.
Let $(\calR^*, \lambda^*)$ be an $\omega$-saturated elementary extension of $(\R, <, +,r^{-\N})$, and let $b \in \lambda^*(\calR^*)$ be such that $b \ll 1$ and $(\calR^*, \lambda^*) \models \phi(b)$.
Let $b'\in \lambda^*(\calR^*)$ be the largest element of $\lambda^*(\calR^*)$ less than $b$.
Then $\neg \phi (b')$ holds, since $m>1$.
Yet we know by Lemma \ref{lambdaQE} that $b$ and $b'$ have the same $\calL_{\lambda}(\R)$-type, a contradiction.
So no such $\phi$ exists, and hence none exists in the language $(<,+,r^{-\N})$ either.
\end{proof}
\end{fact}

Throughout, we say a set $X\subseteq \R$ is \emph{discrete} if for every $x \in X$ there is an open interval $I \subseteq \R$ such that $\{x\} = X \cap I$.

We say that a structure is \emph{d-minimal} if for every $(m+1)$-ary definable set $A$, there is $N \in \N$ such that for each $m$-tuple $x$ the fiber $\{t: (x,t) \in A\}$ either has interior or is the union of at most $N$ discrete sets, as defined in \cite{Mi05}.
The \emph{Cantor-Bendixson derivative} of a set $X \subseteq \R^d$ is $X \setminus X^{iso}$, where $X^{iso}$ are the isolated points of $X$.
The Cantor-Bendixson derivative of $X$ is often denoted by $X'$, and for $\alpha$ an ordinal, let $X^{(\alpha)}$ denote the result of taking the Cantor-Bendixson derivative $\alpha$ times.
The \emph{Cantor-Bendixson rank} of $X$ is the least ordinal $\alpha$ such that $X^{(\alpha)} = X^{(\alpha +1)}$.
If the Cantor-Bendixson rank of $X$ is $\alpha$ and $X^{(\alpha)} = \emptyset$, then we say that $X$ has vanishing Cantor-Bendixson derivative.

\begin{lem}\label{CB-rank}
Suppose that $A \subseteq [0,1]^d$ is $r$-regular and closed.
Then $A$ has vanishing Cantor-Bendixson derivative and finite Cantor-Bendixson rank if and only if $A$ is $r$-sparse.
\begin{proof}
First, we observe that Proposition \ref{sparseDmin} tells us that $A$ is definable in a d-minimal structure, namely $\calR_r$.
Since $A$ is a finite union of discrete sets definable in a d-minimal structure, it has finite Cantor-Bendixson rank; see \cite{Mi05} or \cite{MT18}.
 Moreover, a closed subset of $\R^d$ has vanishing Cantor-Bendixson derivative if and only if it is countable \cite[\S6.B]{Kec}.
The backward direction follows from putting together that $A$ must have vanishing Cantor-Bendixson derivative, and the Cantor-Bendixson rank must be finite.

For the forward direction, we suppose that $A$ is closed but not $r$-sparse.
Then Lemma \ref{TFAE} tells us that $A$ is uncountable, and hence does not have vanishing Cantor-Bendixson derivative by the above ``if and only if'' statement.
\end{proof}
\end{lem}

\begin{lem}\label{isolated}
Suppose that $A \subseteq [0,1]^d$ is $r$-sparse and infinite.
Let $A^{acc}$ denote the accumulation points of $A$, i.e. non-isolated points in $\bar{A}$.
Then $A^{acc}$ has an isolated point.
\end{lem}
\begin{proof}
By Lemma \ref{CB-rank}, we know that $\bar{A}$ has vanishing Cantor-Bendixson derivative, and finite Cantor-Bendixson rank.
This means there is $n \in \N$ such that the $n$-th Cantor-Bendixson derivative is $\emptyset$.
Since $\bar{A}$ is an infinite compact subset of $\R^d$, it must have at least one accumulation point (by the Bolzano-Weierstrass theorem).
Since taking the Cantor-Bendixson derivative removes the isolated points, the first Cantor-Bendixson derivative of $\bar{A}$ is $A^{acc} \neq \emptyset$.
Taking the second Cantor-Bendixson derivative, it cannot be the case that we get $A^{acc}$ again, because this would contradict the fact that $\bar{A}$ has vanishing Cantor-Bendixson derivative.
So we conclude there is at least one isolated point in $A^{acc}$.
\end{proof}

We now establish some notation concerning subwords.
For any $u,v \in ([r]^m)^*$, write $u \subseteq v$ if $u$ is a prefix of $v$.
If it does not hold that $u \subseteq v$, then we write $u \not \subseteq v$.

We are now able to prove the last result needed for Theorem \ref{sparse rN}.
For the rest of this section, set $\calR_{A} :=(\R, <,+,0,1, A)$ with $A$ a unary predicate.

Note that in any linearly ordered structure (with the ordering symbol $<$ in the language) if $X$ is a definable and discrete set, the graph of the predecessor function on $X$ is defined as follows:
\begin{align}\label{pred}
P_X(x)= & y \iff x \in X \land y \in X \land y<x \land \forall z(y<z<x \rightarrow z \not \in X).
\end{align}

\begin{prop} \label{rtoNdefinable}
The set $r^{-\N}$ is $\emptyset$-definable in $\calR_{A}$, whenever $A\subseteq [0,1]^m$ is infinite and $r$-sparse.
\begin{proof}
Since $A$ is $r$-sparse and infinite, there is some $i \in [d]$ such that the projection of $A$ onto the $i$-th coordinate is again infinite and sparse, thus we may assume that $m=1$. 
We note that we may also define the closure of $A$, and thus assume $A$ is closed as well. 
Since we can define the accumulation points of $A$, by  Lemma \ref{CB-rank}, we see that by replacing $A$ by a suitable higher Cantor-Bendixson derivative, we may assume that $A$ is infinite and sparse but has only finitely many accumulation points. 
Finally, since we can define the intersection of $A$ with a small neighborhood of a single accumulation point, we may in fact assume that $A$ has exactly one accumulation point.

By Remark \ref{sparse-description0} we can express $A$ as a finite union of sets of the following form:
\begin{align}\label{sparseA}
\left\{ c_0  + c_1 r^{-\delta_1 n_1} + c_2 r^{-\delta_1 n_1 - \delta_{2} n_{2}}\cdots + c_{d-1} r^{-\delta_1 n_1-\cdots -\delta_{d-1} n_{d-1}} \colon n_1,\ldots ,n_{d-1}\ge 0\right\},
\end{align}

with $d\ge 1$ and $c_1,\ldots ,c_{d-1}$ nonzero rational numbers and $c_0$ rational.  Moreover, since $A$ has only one accumulation point, we must have $d\le 2$ for all such sets that $A$ comprises, since the accumulation points of (\ref{sparseA}) comprise the set
$$
\left\{ c_0  + c_1 r^{-\delta_1 n_1} + c_2 r^{-\delta_1 n_1 - \delta_{2} n_{2}}\cdots + c_{d-2} r^{-\delta_1 n_1-\cdots -\delta_{d-2} n_{d-2}} \colon n_1,\ldots ,n_{d-2}\ge 0\right\},$$ and by assumption $c_1,\dots , c_{d-2}$ are nonzero. Moreover, at least one set of this form has $d=2$, since $A$ is infinite.

It follows that $A$ is the union of a finite set along with a finite number of sets $\{ a_i + b_i r^{-\delta_i \mathbb{N}} \colon i=1,\ldots ,s\}$
with $a_i,b_i$ rational, $b_i$ nonzero, and $\delta_i$ a positive integer for $i=1,\ldots ,s$. We can remove a finite number of rational elements from $A$ and so we may assume that $A$ is precisely this finite union of infinite sets given above. 

Moreover, since each $a_i$ is an accumulation point and $A$ has a single accumulation point, we see that $a_1=a_2=\cdots =a_s$.
Since we can define a translate of $A$ by a rational number and a scaling of $A$ by a positive integer, we may assume without loss of generality that $a_1=\cdots =a_s=0$ and that $b_1,\ldots ,b_s$ are integers with at least one $b_i$ positive.  
Finally, by replacing $A$ by $A\cap [0,\infty)$ we may assume that each $b_i$ is a positive integer.  We may also assume that $r$ is not a power of a smaller positive integer, by enlarging the $\delta_i$'s if necessary.

Now we define the set $B$ consisting of all $x\in (0,1)$ such that $b_i x\in A$ for each $i=1,\ldots ,s$.
We claim that $B=r^{-T}$, where $T$ is an infinite eventually periodic subset of $\mathbb{N}$.  
To see this, observe that by construction for sufficiently large $N$, we have that if 
$r^{-N}$ is in $B$ then so is $r^{-N-\delta}$, where $\delta$ is the lcm of $\delta_1,\ldots ,\delta_s$. 
Since $B\subseteq (0,1)$, we see that it suffices to show that $B\subseteq r^{\mathbb{Z}}$.

We now let $\sim_r$ be the equivalence relation on $\mathbb{Q}\setminus\{0\}$ where two nonzero rational numbers are equivalent if their ratio is an integer power of $r$.  Now suppose towards a contradiction that $B$ is not contained in $r^{-\mathbb{N}}$. Then there is some element $\lambda\in B$, with $\lambda\not\sim_r 1$. 
By construction, for each $i$ there is some index $a(i)$ such that $\lambda  \sim_r b_{a(i)}/b_i$.
Now if we start with $b_1$ and consider the sequence $1, a(1), a(a(1)), \ldots$, we see that at some point we have a cycle, $p_1,\ldots ,p_m$ with $a(p_i)=p_{i+1}$ for $i=1,\ldots, m$, where $p_{m+1}=p_1$.  
Then by telescoping, we have $$\lambda^m \sim_r \prod_{i=1}^m b_{p_{i+1}}/b_{p_i} = 1.$$  
So $\lambda^m\sim_r 1$, and is hence a power of $r$. Since $r$ is not itself a power of a smaller integer, we see that $\lambda\sim_r 1$, a contradiction.  Thus $B=r^{-T}$ for some infinite eventually periodic subset $T$ of $\mathbb{N}$.  

Now let $\ell$ be the smallest positive eventual period of $T$; that is, for all $n\in T$ sufficiently large we have $n+\ell\in T$, but no smaller positive $\ell$ has this property. Then let $0\le i_1 <\cdots < i_q <\ell$ denote the set of numbers $i\in \{0,\ldots ,\ell-1\}$ such that $T\cap (i+\ell \mathbb{N})$ is infinite.  Now we can define the set 
$C$ consisting of elements $x$ such that $r^{-i_1}x,\ldots ,r^{-i_q}x$ are all in $B$.
Notice that $C$ contains $k^{-\ell (\mathbb{N}+m)}$ for some integer $m$.  We claim that up to a finite set, it is exactly this set. To see this, observe that if this is not the case, then by construction $C$ must contain a set of the form $k^{-i +\ell(\mathbb{N}+m')}$ for some $i\in \{1,2,\ldots ,\ell-1\}$ and some $m'>0$. But this means that $i+i_1,\ldots ,i+i_q$ must be a permutation of $i_1,\ldots ,i_q$ mod $\ell$. In particular if $\ell'=\gcd(\ell,i)<\ell$ then the set $T$ must have eventual period $\ell'$, contradicting minimality of $\ell$.  It follows that $C$ has finite symmetric difference with $r^{-\ell \mathbb{N}}$ and so we can define $r^{-\ell\mathbb{N}}$ and hence $r^{-\mathbb{N}}$ by Remark \ref{rmk:33}.

\end{proof}
\end{prop}

If $L\subseteq ([r]^m)^{\omega}$ is sparse and a union of $n$ sets of form $u_1v_1^* \ldots u_{d-1}v_{d-1}^* u_d v_d^{\omega}$, let $W_L:= \{u_{1,1},v_{1,1}, \ldots ,u_{m,d},v_{m,d}\}$ witness this, and be such that for all $i\leq d$ we know $v_{i,1}, \ldots ,v_{i,d}$ are non-empty strings.
Let $\ell_L$ be the least common multiple of $\{|w|: w\in W_L\}$.

\begin{cor}[Theorem \ref{sparse rN}]\label{sparseD}
If $A$ is infinite and $r$-sparse and $L \subseteq ([r]^m)^{\omega}$ is such that $\nu_r(L)=A$, then there is a reduct $\tilde{\calR}$ of $\calR_{r,\ell_L}$ that expands $(\R,<,+,r^{-\N})$ and defines the same sets (without parameters) as $\calR_A$.
Hence $\calR_A$ is d-minimal if $A$ is $r$-sparse.
\begin{proof}
By Proposition \ref{rtoNdefinable}, if $A$ is $r$-sparse then the set $r^{-\N}$ is $\emptyset$-definable in the language $$(<,+,0,1,A),$$ and similarly by Proposition \ref{sparseDmin} we know such an $A$ is $\emptyset$-definable in the language $(<,+,0,r^{-\ell_L \N})$.
Hence $\calR_A$ defines the same sets as some reduct of $\calR_{r,\ell_L}$, and that reduct includes all sets definable with $(<,+,r^{-\N})$.
Moreover, the work of van den Dries \cite{vdD85} establishes that the structure $\calR_{r,\ell_L}$ is d-minimal, so $\calR_A$ is d-minimal as well.
\end{proof}
\end{cor}

We can characterize $r$-sparse sets even more simply as follows.

\begin{cor}
A subset of $\R^m$ is $r$-sparse if and only if it is definable in $\calR_{r, \ell}$ for some integer $\ell$. 
\end{cor}

The following might be relevant to model theorists.

\begin{cor}\label{sparseNIP}
The structure $\calR_A$ has NIP whenever $A$ is $r$-sparse.
\begin{proof}
This is immediate from Theorem \ref{sparse rN} and results from \cite{GH11}, in which the authors show that certain d-minimal structures (which include our $\calR_A$) have NIP.
\end{proof}
\end{cor}

\section{A fractal geometry and tameness dichotomy}

In this section, we consider the structure $\calR_{A} :=(\R, <,0, +, A)$ for which we now mandate that $A$ be an $r$-regular subset of $[0,1]^d$, but not necessarily $r$-sparse.
We will start by proving some key facts about the isolated points of $A$, denoted $A^{iso}$, and the subset $A^{\rm{ctbl}}$
of $A$ consisting of points $x \in A$ for which there exists an interval $I \ni x$ such that $I \cap A$ is countable.
Throughout this section, let $L\subseteq ([r]^d)^{\omega}$ be such that $\nu_r(L)=A$, let $L^{iso} \subseteq ([r]^d)^{\omega}$ be such that $\nu_r(L^{iso})=A^{iso}$, and let $L^{\rm{ctbl}} \subseteq ([r]^d)^{\omega}$ be such that $\nu_r(L^{\rm{ctbl}})=A^{\rm{ctbl}}$.
We use throughout that by Theorem \ref{buchi} we can write $L$ as the finite union of sets of the form $V_1W_1^{\omega}, \ldots ,V_nW_n^{\omega}$ where each $V_i$ and $W_i$ are in $([r]^d)^*$.

We now introduce some notation in the context of $r$-representations.
Given an element $x \in ([r]^m)^{\omega}$, 
let $x_{[a,b]}$ with $a \in \N$ and $b \in \N\cup \{\infty \}$ denote the substring of $x$ starting at index $a \in \omega$ and ending with index $b \in \omega$.
When $b=\infty$, then $x_{[a,\infty)}$ denotes the tail sequence of $x$ starting at index $a$, 
and it is again an element of $([r]^m)^{\omega}$.
For $u,v \in ([r]^m)^{\omega}$, we say that $u\simeq_k v$ precisely if $u_{[0,k]}=v_{[0,k]}$.
Given a language $L\subseteq ([r]^m)^{\omega}$, let $[u]_{\simeq_k,L}$ denote $\{v \in L: u_{[0,k]}=v_{[0,k]}\}$.
We say that $x,y \in ([r]^m)^{\omega}$ are in the same \emph{tail equivalence class} if there exist $k,\ell \in \N$ such that $x_{[k,\infty)}=y_{[\ell,\infty)}$.

\begin{lem}\label{ctbl}
Let $L \subseteq  [r]^{\omega}$ be regular.
For every element $x \in L$, the element $\nu_r(x)$ is not in the closure of any uncountable V-W component of $L$ if and only if $x \in L^{iso} \cup L^{\rm{ctbl}}$.
Moreover, the elements of $L^{\rm{ctbl}}$ and $L^{iso}$ are in one of finitely many tail-equivalence classes.
\begin{proof}

For the backwards direction, we consider the contrapositive; suppose $\nu_r(x)$ is in the closure of some uncountable V-W component of $L$, and we will demonstrate that $x \not \in L^{iso} \cup L^{\rm{ctbl}}$.
Fix one such component $VW^{\omega}$.
Since $x$ is in the closure of $VW^{\omega}$, it has a $r$-representation that is accepted by the automaton-theoretic closure of $VW^{\omega}$.

We define the \emph{limit language} of a (finite or infinite) language $L$ on alphabet $\Sigma$ to be the set of infinite words $w \in \Sigma^{\omega}$ such that infinitely many prefixes of $w$ are also a prefix of a word in $L$.
We observe that $\overline{VW^{\omega}}$, the closure of $VW^{\omega}$, is the limit language the of $V$ plus $V$ concatenated on the left of the limit language of $W^{\omega}$, which is simply $VW^{\omega}$ since an $\omega$-power is its own limit language.
So $\overline{VW^{\omega}} = V^{\omega} \cup VW^{\omega}$.

For any open neighborhood $U$ containing $\nu_r(x)$, we know there is some prefix $x'$ of $x$ such that $x'$ is in $VW^*$ and the image under $\nu_r$ of every element of $VW^{\omega}$ with prefix $x'$ is also in $U$.
Since we assume $VW^{\omega}$ is uncountable, it is necessarily the case that $W^{\omega}$ is uncountable.
Hence $\nu_r(x'W^{\omega}) \subseteq U$ is uncountable and contains $\nu_r(x)$, witnessing that $x$ cannot be in $L^{\rm{ctbl}}$ nor in $L^{iso}$ if it is in an uncountable $VW^{\omega}$ component, or the closure thereof.

Conversely, suppose that $x \in L$ but $x$ is not in the closure of any uncountable component.
Then for a small enough neighborhood $U$ containing $\nu_r(x)$, we conclude that for each  V-W component of $L$, $U \cap \nu_r(VW^{\omega}) \neq \emptyset$ if and only if $VW^{\omega}$ is a countable component.
Hence $U \cap \nu_r(L) \subseteq \nu_r(\bigcup_{i}V_iW_i^{\omega})$ where $i$ ranges over only the indices of countable V-W components.
Thus $x \in \nu_r(\bigcup_{i}V_iW_i^{\omega})$ where $i$ ranges over only the indices of countable V-W components, implying that $x \in L^{iso}$ and $L^{\rm{ctbl}}$.

It is also immediate that $L^{iso}$ and $L^{\rm{ctbl}}$ contain elements from only finitely many tail-equivalence class, as shown in Proposition \ref{sparseequiv}.
\end{proof}
\end{lem}
 
We will see next that if $VW^{\omega}$ is uncountable, then $\nu_r(VW^{\omega})$ has infinite Cantor-Bendixson rank, and that the set $A^{\rm{ctbl}}$ has finite Cantor-Bendixson rank.

\begin{rmk}
Suppose that $A$ is closed, and $x \in L^{\rm{ctbl}} \setminus L^{iso}$.  
Then there exists some $k \in \N$ and some $r$-sparse subset $B$ such that $\nu_r(x_{[0,k]}[r]^{\omega} \cap L) \subseteq B$.
\begin{proof}
Let $U \ni x$ be an open set that witnesses $x \in L^{\rm{ctbl}} \setminus L^{iso}$, i.e.\ $A \cap U$ is countably infinite and contains $x$.
Then there exists a least $k \in \N$ such if $y \in L$ has prefix $x_{[0,k]}$, then $\nu_r(y) \in U$.
Without loss of generality, we may choose $U$ such that $\overline{U \cap A}= \nu_r(x_{[0,k]}[r]^{\omega} \cap L)$.

We observe that $\nu_r(x_{[0,k]}[r]^{\omega} \cap L)$ is closed since $\nu_r(x_{[0,k]}[r]^{\omega})$ and $\nu_r(L)$ both are.
Since $\overline{U \cap A}$ is countable and infinite by hypothesis, it follows from Lemma \ref{TFAE}
that $\nu_r(x_{[0,k]}[r]^{\omega} \cap L)$ must be contained in an $r$-sparse set.
\end{proof}
\end{rmk}

There is one last piece to put in place before we can employ these lemmas to prove the primary result of this section.
The final step is to show that for $A$ a closed $r$-regular set, not only is the set of points that accumulate to any given $x \in A^{\rm{ctbl}}$ an $r$-sparse set, i.e. one with finite Cantor-Bendixson rank and vanishing Cantor-Bendixson derivative, but also there is a uniform bound on the Cantor-Bendixson rank of the set of elements in $A$ that are sufficiently close to $x \in A^{\rm{ctbl}}$.
Here, ``sufficiently close'' means there is an open interval containing $x$, and whose diameter depends on $x$, in which the intersection with $A$ has a Cantor-Bendixson rank less than the uniform bound, which does not depend on $x$.

\begin{lem}\label{uniformCB}
Suppose that $L$ is closed.
Then there exists $N \in \N$ such that for every $x \in L^{\rm{ctbl}}$ there is some $k \in \N$ such that the Cantor-Bendixson rank of $\nu_r([x]_{\simeq_k,L})$ is at most $N$.
\begin{proof} 
Since the Cantor-Bendixson rank of a set $B$ of real numbers is equal to the Cantor-Bendixson rank of $cB$ for a nonzero real number $c$ and to the rank of $B+a$ for a real number $a$, we see that the Cantor-Bendixson rank of $\nu_r([x]_{\simeq_k,L})$ does not depend on the choice of equivalence class representative $x$ but only on the states in each V-W component reach by the length-$k$ prefix of $x$.  As $x$ runs over elements of $L^{\rm{ctbl}}$, there are only finitely many possible states reached by prefixes of $x$ within V-W components, which gives an upper bound on the Cantor-Bendixson ranks that does not depend on $k$.  

Recall that $[x]_{\simeq_k,L}=\{y\in L: y_{[0,k]}=x_{[0,k]}\}$.
We suppose not, i.e. that there exists a sequence $(x_n)_{n \in \N}$ such that for each $n \in \N$ and each $k \in \N$ the set $[x_n]_{\simeq k,L}$ has Cantor-Bendixson rank greater than $n$.
We can decompose $L$ into $E= \overline{ \operatorname{int}(\nu_r(L))}$, the closure of the interior of $L$, and $L \setminus E$, where the latter is nowhere dense on an open subset of $[0,1]$.
Since the Cantor-Bendixson derivative on an interval is itself, we may disregard $E$ when bounding the Cantor-Bendixson rank of $\nu_r([x]_{\simeq_k,L})$ for $x \in L$.
So, without loss of generality, we may assume $L = L \setminus E$, i.\ e.\ that $L$ has no interior.

By the pigeonhole principle, we may take all of the $x_n$'s to be in the same V-W component of $L$.
Since we have assumed that $L$ has no interior, for each $x_n$ there is $k_n \in \N$ such that $[x_n]_{\simeq_{k_n},L}$ is $r$-sparse. 
By Lemma \ref{TFAE}, this means the Cantor-Bendixson rank of $[x_n]_{\simeq_{k_n},L}$ is some finite $m_n$ with $m_n >n$ by hypothesis.
Applying 
sequential compactness of the $[0,1]$ interval
to the sequence $(x_n)_{n \in \N}$, 
we conclude that there is an infinite subsequence of $(x_n)_{n \in \N}$ such that for $n',\ell >n$ it holds that $x_{n[0,\ell]} = x_{n'[0,\ell]}$.
Hence we may construct an infinite subsequence $(x_{n(\ell )})_{\ell \in \N}$ in which $x_{n(\ell +1)[0,\ell +1]}$ extends $x_{n(\ell )[0,\ell ]}$ for each $\ell \in \N$.
Without loss of generality, let $(x_n)_{n \in \N}$ now denote this subsequence.

Since $x_{n+1}$ and $x_n$ have a shared prefix of at least length $n$ for each $n\in \N$, we will now establish that that the numbers $k_n,k_{n+1}$, witnessing respectively that $[x_n]_{\simeq_{k_n},L}$ and $[x_{n+1}]_{\simeq_{k_{n+1}},L}$ are $r$-sparse, satisfies $k_{n+1} > k_n$.
Otherwise, the sequence of $k_n$'s would stabilize, implying that eventually $[x_n]_{\simeq_{k_n},L}=[x_{n'}]_{\simeq_{k_{n'}},L}$ for sufficiently large $n'>n$, which in turn implies the Cantor-Bendixson rank of these sets stabilizes.
We reach a contradiction because by definition of Cantor-Bendixson rank, removing the isolated points of $[x_n]_{\simeq_{k_n},L}$ iteratively $m_n:= CB([x_n]_{\simeq_{k_n},L})$ times should yield the empty set.
Yet by definition we know that for sufficiently large $n'>m_n>n$ we have 
$$[x_{n'}]_{\simeq_{k_{n'}},L} = \{y \in L: y_{[0,k_{n'}]}=x_{n'[0,k_{n'}]}\} \subseteq \{y \in L: y_{[0,k_n]}=x_{n[0,k_n]}\}=[x_n]_{\simeq_{k_n},L}$$ 
since $k_n \leq k_{n'}$ ensures that any such $y_{[0,k_{n'}]}$ contains the prefix $x_{n[0,k_n]}$.
This would mean that removing the isolated points of $[x_{n'}]_{\simeq_{k_{n'}},L}$ iteratively $m_n$ times yields the empty set, contradicting that $m_{n'}>n'$ is strictly greater than $m_n$.

So we conclude that there is some $N \in \N$ such that for each $x \in L^{\rm{ctbl}}$ there is eventually a $k \in \N$ such that the Cantor-Bendixson rank of $[x]_{\simeq_k,L}$ stabilizes to a number less than or equal to $N$.
\end{proof}
\end{lem}


Recall that within Euclidean space a \emph{Cantor set}, used in the most general sense, is a nonempty subset of $\R^d$ that is compact, has no isolated points, and has no interior.

\begin{thm}[Theorem \ref{tame}]\label{hdimCantor}
If $A$ is a bounded $r$-regular subset of $[0,1]^d$ and there is some $i \in [d]$ such that $0<d_H(\overline{\pi_i(A)})<1$, then there exists a unary Cantor set $X$ definable in $\calR_A = (\R,<,+,0,1,A)$.
\begin{proof}
Let $i \in [d]$ be such that $A_i:=\overline{\pi_i(A)}$ witness the hypothesis of the claim, i.e. $0<d_H(A_i)<1$.
Note that $A_i$ is compact since $A$ is bounded.
We will definably remove a countable set of points from $A_i$ such that the remaining set $X$ is closed and has no isolated points.
Then because $d_H(A_i)$ is between $0$ and $1$, we conclude the same is true for $X$.

Recall that we use $A_i^{\rm{ctbl}}$ to denote the subset of $A_i$ such that $x \in A_i^{\rm{ctbl}}$ precisely if there is some open interval $I$ containing $x$ such that $A_i \cap I$ is countable.
By Lemma \ref{uniformCB}, we know that there is some $N \in \N$ such that if we strip away the isolated points of $A_i^{iso} \cup A_i^{\rm{ctbl}}$ at least $N$ times, we are left with the empty set.
By Lemma \ref{ctbl}, we know that $A_i^{iso} \cup A_i^{\rm{ctbl}}$ is the part of $A_i$ whose $r$-representations are contained in the complement of the closure of the uncountable $V_kW_k^{\omega} $ components of $L$. 
Call the image under $\nu_r$ of the closure of these uncountable components $X$.

Hence we can write $A_i$ as the disjoint union of $A_i^{iso} \cup A_i^{\rm{ctbl}}$ and $X$.
When we take the $N$-th Cantor-Bendixson derivative of $A_i$, i.e. strip away the isolated points iteratively $N$ times, what remains is $X$, 
since $X$ being a closed set without isolated points makes it a perfect set, and hence its own Cantor-Bendixson derivative.
Notably, taking the Cantor-Bendixson derivative $N$ times is definable in the language $(<,+,0,1)$, so we can isolate $X$ definably.

Lastly, to see that $X$ is a Cantor set we determine that it cannot have interior.  If so, it would have a subset with Hausdorff dimension $1$, since the Hausdorff $1$-measure is positive whenever Lebesgue measure is positive.
\end{proof}
\end{thm}

\begin{cor}[Theorem \ref{r-reg dichotomy}]\label{r-reg tame}
For $A\subseteq [0,1]^d$ an $r$-regular set such that $d_H(\overline{\pi_i(A)})<1$ for all $i \in [d]$, the following are equivalent:
\begin{enumerate}
\item{$A$ is $r$-sparse;}
\item{$d_H(\overline{\pi_i(A)})=0$ for all $i\in [d]$;}
\item{$\calR_A$ is d-minimal;}
\item{$\calR_A$ has NIP;}
\item{$\calR_A$ has NTP$_2$.}
\end{enumerate}
\begin{proof}
As above, set $A_i:=\pi_i(A)$ for $i \in [d]$.
The equivalence of $(1)$ and $(2)$ follows from Lemma \ref{dimH0} and Remark \ref{sparse-closure}.
To see that $(2)$ implies $(3)$, we note that if $d_H(\overline{A_i})=0$ for each $i \in [d]$, then by Lemma \ref{dimH0} we know that $\overline{A_i}$ is $r$-sparse, and hence countable, for each $i \in [d]$.
We know $\overline{A} \subseteq \overline{A_1}\times \cdots \times \overline{A_d}$ since coordinate projection is a continuous open map, and the latter set is countable.
We conclude that $A$ is countable, thus by Lemma \ref{TFAE} the set $A$ is $r$-sparse.
Hence by Theorem \ref{sparseD}, we conclude that $\calR_A$ is d-minimal.

For $(3)$ implies $(4)$, this follows from \cite{GH11}, in which that authors show that certain d-minimal structures, which include $\calR_A$, have NIP.
That $(4)$ implies $(5)$ follows from the definitions of NIP and NTP$_2$.

Finally, to see that $(5)$ implies $(2)$, we show the contrapositive.
Suppose that $d_H(\overline{\pi_i(A)})>0$ for some $i\in [d]$.
By hypothesis, $d_H(\overline{\pi_i(A)})<1$ for all $i\in [d]$, so by Theorem \ref{tame} we conclude that a Cantor set is definable in $\calR_A$.
By \cite{HW19}, this also means the structure interprets the monadic second order theory of the natural numbers with the successor function, which has TP$_2$ (and thus does not have NTP$_2$).
\end{proof}
\end{cor}

\section{Simultaneous Sparsity in multiplicatively independent bases}
In this section, we consider the intersection of two sparse sets that are regular with respect to two multiplicatively independent bases.
We will show that sparse $k$- and $\ell$-regular sets have finite intersection when $k$ and $\ell$ are multiplicatively independent and give upper bounds on the size of the intersection in terms of the accepting B\"uchi automata for these sets.  We note that this work shares some overlap with the paper \cite{AB}.

\begin{lemma} \label{rem:sparse}
Let $k\ge 2$ be a natural number and let $S$ be a non-empty sparse $k$-regular subset of $[0,1]$. Then $S$ is a finite union of sets $S_1\cup S_2\cup \cdots \cup S_m$ such that for each $i\in \{1,2,\ldots ,m\}$ we have:
\begin{enumerate}
    \item[(i)] there exists $s=s_i \ge 0$, and $c_0,\ldots ,c_s\in \mathbb{Q}$ with the property that for some $\ell>0$ we have $(k^{\ell}-1)c_j\in \mathbb{Z}$ for $j=0,\ldots ,s$ and $c_0 + c_1 + \cdots + c_s \in \mathbb{Z}_{\ge 0}$;
    \item[(ii)] there exist positive integers $\delta_1,\ldots ,\delta_s$
\end{enumerate}
such that the we have the following containment:
\begin{equation} \label{eq: form}
S_i \subseteq \left\{ c_0 k^q  + c_1 k^{q+\delta_s n_s} + c_2 k^{q+\delta_s n_s + \delta_{s-1} n_{s-1}}\cdots + c_s k^{q+\delta_s n_s+\cdots +\delta_1 n_1} \colon n_1,\ldots ,n_s\ge 0, q \in -\mathbb{N}\right\}.\end{equation}
\end{lemma}
\begin{proof} We know that a sparse $k$-regular subset of $[0,1]$ in a finite union of sets $S_1,\ldots ,S_m$ with each $S_i$ a set of the form
$$\{[ _{\bullet} u_0 v_0^{n_1} u_1 v_1^{n_1} \cdots v_s^{n_s} u_s v_{s+1}^{\omega}]_k \colon n_1,\ldots ,n_s\ge 0\}.$$
Notice this set $S_i$ is contained in the union of all sets of the form $k^{-q} T_i$ with $q\ge 0$, where $T_i$ is the set of rational numbers of the form
$$\{[u_0 v_0^{n_1} u_1 v_1^{n_1} \cdots v_s^{n_s} u_s {_{\bullet}} v_{s+1}^{\omega}\}.$$
By a computation that is analogous to the one performed in \cite[Remark 3.4]{AB1}, the set of natural numbers of the form
$[u_0 v_0^{n_1} u_1 v_1^{n_1} \cdots v_s^{n_s} u_s]_k$ is just the set
$$\{ c_0   + c_1 k^{\delta_s n_s} + c_2 k^{\delta_s n_s + \delta_{s-1} n_{s-1}}\cdots + c_s k^{\delta_s n_s+\cdots +\delta_1 n_1}\colon n_1,\ldots ,n_s\ge 0\}$$ for some $c_i, \delta_i$ as above.  Since $[ _{\bullet} v_{s+1}^{\omega}]_k$ is a rational number of the form $c/(k^p-1)$ for some positive integer $c$ and some $p\ge 1$, we may replace $c_0$ by 
$c_0+[ _{\bullet} v_{s+1}^{\omega}]_k$ and we get the desired result.  
\end{proof}
We call a $k$-regular subset of $[0,1]$ of the form $$\{[ _{\bullet} u_0 v_0^{n_1} u_1 v_1^{n_1} \cdots v_s^{n_s} u_s v_{s+1}^{\omega}]_k \colon n_1,\ldots ,n_s\ge 0\}$$ a \emph{simple} sparse $k$-regular set of \emph{length} $s$.

The following result is primarily a consequence of Theorem \ref{thm:eg}.
\begin{lem} Let $k$ and $\ell$ be multiplicatively independent integers, let $n,m\ge 1$, and let $a_0,\ldots, a_{n+m+1}$ be nonzero rational numbers, then there are at most 
\begin{equation}
\exp(3(6(n+m))^{3n+3m})
\end{equation}
solutions to the equation
\begin{equation}
a_0X_0+\cdots +a_{n+m+1} X_{n+m+1}=0 \label{eq:abc}
\end{equation}
in which each $X_i$ is a power of $k$ for $i=0,\ldots ,n$ and each $X_i$ is a power of $\ell$ for $a_{n+1},\ldots , a_{n+m+1}$.
\label{lem:nondegen}
\end{lem}
\begin{proof}
    Letting $a_i' = -a_i/a_0$ for $i=0,\ldots ,m$, we see that after dividing Equation (\ref{eq:abc}) by $a_0X_0$ that a nondegenerate solution to 
Equation (\ref{eq:abc}) of the desired form corresponds to a nondegenerate solution $(Y_1,\ldots ,Y_{n+m+1})$ to 
the equation $$\sum_{i=1}^{n+m+1} a_i' Y_i = 1$$ in which $Y_1,\ldots, Y_n$ are integer powers of $k$, and $Y_{n+1},\ldots ,Y_{n+m+1}$ are of the form $k^c$ times an integer power of $\ell$ for some fixed $c$.  (Here $k^c=1/X_0$.) By Theorem \ref{thm:eg} there are at most 
$\exp(3(6(n+m+1))^{3n+3m+1})$ nondegenerate solutions to this equation.  Since the integer power $k^c$ is uniquely determined by $Y_{n+1}$, we can recover our original solution $(X_0,\ldots ,X_{n+m+1})$ from $(Y_1,\ldots ,Y_{n+m+1})$.  The result follows.
\end{proof}
\begin{lem} Let $k$ and $\ell$ be multiplicatively independent integers, let $n,m\ge 1$, and let $a_0,\ldots, a_{n+m+1}$ be nonzero rational numbers, then there are at most 
\begin{equation}
(n+m+2)^{n+m+2}\exp(3(6(n+m+1))^{3(n+m+1)})
\end{equation}
solutions to the equation
$$
a_0X_0+\cdots +a_{n+m+1} X_{n+m+1}=0$$
in which each $X_i$ is a power of $k$ for $i=0,\ldots ,n$ and each $X_i$ is a power of $\ell$ for $a_{n+1},\ldots , a_{n+m+1}$ and in which no non-trivial subsum of either
$a_0X_0+\cdots +a_n X_n$ or $a_{n+1} X_{n+1}+\cdots + a_{n+m+1} X_{n+m+1}$ vanishes.
\label{lem:nondegen2}
\end{lem}
\begin{proof} Let $\Gamma\le \mathbb{Q}^*$ be the rank two group generated by $k$ and $\ell$.
For each solution to
$$a_0X_0+\cdots +a_{n+m+1} X_{n+m+1}=0$$ such that no subsum of either $a_0X_0+\cdots +a_n X_n$ or $a_{n+1} X_{n+1}+\cdots + a_{n+m+1} X_{n+m+1}$ vanishes, we can form a set partition 
of the set of variables $\{X_0,\ldots, X_{n+m+1}\}$ into disjoint non-empty subsets $V_1,\ldots ,V_s$ such that the subsum 
corresponding to the variables in each $V_i$ vanishes and no proper subsum vanishes. By assumption, each $V_i$ must intersect 
both $\{X_0,\ldots ,X_n\}$ and $\{X_{n+1},\ldots ,X_{n+m+1}\}$ non-trivially.  Then by Lemma \ref{lem:nondegen}, there are at 
most $\exp(3(6(|V_i|-1))^{3(|V_i|-1)})$
nondegenerate solutions to the subsum equation
$$\sum_{X_j\in V_i} a_j X_j  = 0$$ with each $X_j$ a power of $k$ for $j\le n$ and a power of $\ell$ for $j>n$.

Thus for the set partition as above with parts $V_1,V_2,\ldots, V_s$ we have at most 
$$\prod_{i=1}^s \exp(3(6(|V_i|-1))^{3(|V_i|-1)}) \le \exp(3(6(n+m+1))^{3(n+m+1)})$$
solutions, where the last step follows from straightforward estimates for the exponential function.

Since a set partition of $\{0,\ldots ,n+m+1\}$ with $s$ parts naturally gives rise to a surjective map from $\{0,\ldots ,n+m+1\}$ to $\{1,\ldots ,s\}$, and since $s\le n+m+2$, we see there are at most $(n+m+2)^{n+m+2}$ possible set partitions.  
Putting this together, we get the desired bound.
\end{proof}

\begin{prop} \label{sets} Let $k$ and $\ell$ be multiplicatively independent positive integers and let 
$S$ be a sparse $k$-regular subset of $[0,1]$ of the form
$$\{ {_{\bullet}} [v_0 w_1^* v_1 w_2^* \cdots v_{s} w_s^* v_{s+1} w_{s+1}^{\omega} ]_k\}$$ and let 
$T$ be sparse $\ell$-automatic set of the form
 $$\{[ {_{\bullet}} u_0 y_1^* u_1 y_2^* \cdots u_{t} y_t^* u_{t+1} y_{t+1}^{\omega}]_{\ell}\}.$$ 
 Then $$\# S\cap T \le 2^{s+t+1} \cdot (s+t+2)^{s+t+2}\exp(3(6(s+t+1))^{3(s+t+1)}).$$
\end{prop}
\begin{proof}
By Lemma \ref{rem:sparse} we have that $S$ is contained in a set of the form
 $$\left\{c_0 k^{-q} + c_1 k^{-q+\delta_s n_s} + c_2 k^{-q+\delta_s n_s + \delta_{s-1} n_{s-1}}\cdots + c_s k^{-q+\delta_s n_s+\cdots +\delta_1 n_1}  \colon n_1,\ldots ,n_s, q\ge 0\right\},$$ where $c_0,\ldots ,c_s$ are rational numbers. 
 Similarly, $T$ is contained in a set of  the form
  $$\left\{ d_0\ell^{-q'} + d_1 \ell^{-q'+\delta_t' m_t} + d_2 \ell^{-q'+\delta_t' m_s + \delta_{t-1}' m_{t-1}}\cdots + d_t \ell^{-q'+\delta_t'm_t+\cdots +\delta_1' m_1}\colon m_1,\ldots, m_t,q'\ge 0\right\},$$ where $d_0,\ldots, d_t$ are rational numbers.
  
Then an element in $S\cap T$ corresponds to a solution to the equation 
\begin{equation}
\label{eq:S}
d_0 X_0+\cdots +d_t X_t + d_{t+1} X_{t+1}+\cdots  + d_{t+s+1} X_{t+s+1}=0,
\end{equation}
where $X_0=\ell^{-q'}, X_1=\ell^{-q'+\delta_t' m_t},\ldots , X_t= \ell^{-q'+\delta_t' m_t+\cdots +\delta_1' m_1}$,
$X_{t+1}=k^{-q}$, \ldots , $X_{t+s+1}=k^{-q+\delta_s n_s+\cdots +\delta_1 n_1}$, and where we take $d_{t+j}=-c_{j-1}$ for $j=1,\ldots ,s+1$. Moreover, the element in the intersection in this case is given by $$A:=d_0 X_0+\cdots + d_t X_t = -(d_{t+1} X_{t+1}+\cdots + d_{t+s+1} X_{t+s+1}).$$

Since we are only concerned about the quantity $A$, we may remove a maximal vanishing subsum of $d_0 X_0+\cdots + d_t X_t$ and a maximal vanishing subsum of $d_{t+1} X_{t+1}+\cdots  d_{t+s+1} X_{t+s+1}$ and assume that both subsums are nondegenerate.

Then since there are at most $2^{t+s+1}$ subsets of $\{X_1,\ldots,X_{t+s+1}\}$ there are at most $2^{t+s+1}$ choices for the maximal vanishing subsets we remove.  Once we fix these subsets, Lemma \ref{lem:nondegen2} gives there are at most
$$(s+t+2)^{s+t+2}\exp(3(6(s+t+1))^{3(s+t+1)})$$
solutions to the resulting equation of the desired form.
Thus there are at most $$2^{s+t+1} \cdot (s+t+2)^{s+t+2}\exp(3(6(s+t+1))^{3(s+t+1)})$$ solutions to Equation (\ref{eq:S}) of the required form.
\end{proof}
\begin{rem} In general if $S$ and $T$ are respectively a sparse $k$-regular subset of $[0,1]$ and a sparse $\ell$-regular subset of $[0,1]$, then there are $p,q\ge 1$ such that $S$ is the union of $p$ simple sparse sets of lengths $s_1,\ldots ,s_p$ and $T$ is the union of $q$ simple sparse sets of lengths $t_1,\ldots ,t_q$.  Moreover, we can determine $p$, $q$, $s_1,\ldots, s_p, t_1,\ldots ,t_q$ from the B\"uchi automata that accept $S$ and $T$. 
In fact, if $a$ and $b$ are respectively the number of states in the B\"uchi automata that accept $S$ and $T$, we have $p\le k^a$, $q\le \ell^b$,
$s:=\max(s_1,\ldots ,s_p)\le a$, and $t:=\max(t_1,\ldots ,t_q)\le b$.
Proposition \ref{sets} then gives a computable bound for $S\cap T$ purely in terms of the number of states in the B\"uchi automata accepting the sets $S$ and $T$.  
\label{rem:gen}
\end{rem}

\begin{thm} Let $k$ and $\ell$ be multiplicatively independent positive integers and let $S$ and $T$ be respectively $k$- and $\ell$-regular sparse subsets of $[0,1]^d$, with $d\ge 1$.  Then $S\cap T$ is finite.  \label{thm:genint}
\end{thm}
\begin{proof} Let $\pi_i: [0,1]^d \to [0,1]$ denote the projection onto the $i$-th coordinate for $i=1,\ldots ,d$.  Then $\pi_i(S)$ is a sparse $k$-regular subset of $[0,1]$ and $\pi_i(T)$ is a sparse $\ell$-regular subset of $[0,1]$.  
It follows that $Z_i:=\pi_i(S)\cap \pi_i(T)$ is a finite set by Remark \ref{rem:gen}.  Then since $S\cap T\subseteq Z_1\times Z_2\times \cdots \times Z_d$, we see that $S\cap T$ is finite.
\end{proof}
\begin{rem} From the number of states in trim B\"uchi automata that accept $S$ and $T$, we can obtain bounds for the intersection $\pi_i(S)\cap \pi_i(T)$ for $i=1,\ldots ,d$ by Remark \ref{rem:gen} and so we can determine an upper bound for $S\cap T$ in terms of this data.
\end{rem}


\begin{thebibliography}{BG et al.20}

\bibitem[AB21]{AB1}
S. Albayrak and J. P. Bell,
\newblock A refinement of Christol's theorem for algebraic power series.
\emph{Math. Z.}, {\bf 300} (2022), no. 3, 2265–-2288.

\bibitem[AB23]{AB}
 S. Albayrak and J. P. Bell,
\newblock Quantitative estimates for the size of an intersection of sparse automatic sets.
\emph{Theoret. Comput. Sci.}, {\bf 977} (2023), Paper No. 114144, 11 pp.

\bibitem[AS03]{AS} 
J.-P. Allouche and J. Shallit,
\newblock Automatic Sequences. Theory, applications, generalizations.
\newblock \emph{Cambridge University Press, Cambridge}, 2003. 

\bibitem[BGM20]{BGM20}
J. P. Bell, D. Ghioca and R. Moosa,
\newblock Effective isotrivial Mordell-Lang in positive characteristic.
To appear in \emph{Amer. J. Math.}


\bibitem[BHS18]{BHS} 
J. P. Bell, K. Hare and J. Shallit,
\newblock When is an automatic set an additive basis?
\newblock \emph{Proc. Amer. Math. Soc. Ser. B}, {\bf 5} (2018), 50--63.

\bibitem[BM19]{BM19}
J.~Bell and R.~Moosa,
\newblock $F$-sets and finite automata.
\newblock {\em J. Th\'eor. Nombres Bordeaux}, 31(1),101--130, 2019.

\bibitem[BG et al.20]{BG et al.20}
A.~Block Gorman, P.~Hieronymi, E.~Kaplan, R.~Meng, E.~Walsberg, Z.~Wang, Z.~Xiong, and H.~Yang,
\newblock Continuous regular functions, 
\newblock {\em Log. Methods Comput. Sci.}, 16(1): 17:1--17:24, 2020.

\bibitem[BGS23]{BlockGormanSchulz}
\newblock Fractal dimensions of $k$-automatic sets.
\newblock {\em The Journal of Symbolic Logic}, (55), 1--30, 2023.

\bibitem[BRW98]{BRW98}
B.~Boigelot, S.~Rassart, and P.~Wolper, 
\newblock On the expressiveness of real and integer arithmetic automata (extended abstract). 
\newblock {\em Proceedings of the 25th International Colloquium on Automata, Languages and Programming}
(London, UK), ICALP '98, Springer-Verlag, pp. 152--163, 1998.

\bibitem[B\"uc62]{B62}
J.~B\"uchi,
\newblock On a decision method in restricted second order arithmetic.
\newblock {\em Logic, Methodology and Philosophy of Science} (Proc. 1960 Internat. Congr.), Stanford Univ. Press, Stanford, Calif., pp. 1--11, 1962.

\bibitem[CHM02]{CHM02}
C. A. Cabrelli and K. E. Hare and U. M. Molter,
\newblock Sums of Cantor sets yielding an interval.
\newblock \emph{J. Aust. Math. Soc.}, {\bf 73} (2002), 405--418.

\bibitem[CLR15]{CLR15}
{\'E}.~Charlier, J.~Leroy, and M.~Rigo,
\newblock An analogue of {C}obham's theorem for graph directed iterated function systems.
\newblock {\em Adv. Math.}, 280:86--120, 2015.

\bibitem[Der07]{Derksen} 
H. Derksen,
\newblock A Skolem-Mahler-Lech theorem in positive characteristic and finite automata.
\newblock \emph{Invent. Math.} {\bf 168} (2007), no. 1. 175--224.

\bibitem[vdD85]{vdD85}
L. van den Dries,
\newblock The field of reals with a predicate for the powers of two. 
\newblock {\em Manuscripta Math.}, 54, 187--195, 1985.

\bibitem[EG15]{EG}
J.-H. Evertse and K. Gy\H{o}ry,
Unit equations in Diophantine number theory.
\emph{Cambridge Studies in Advanced Mathematics}, 146.
\newblock Cambridge University Press, Cambridge, 2015.

\bibitem[ESS02]{ESS02}
J.-H. Evertse, H.~P.~Schlickewei, and W.~M.~Schmidt,
Linear equations in variables which lie in a multiplicative group.
\emph{ Ann.\ of Math.} \textbf{155} (2002), 807--836.

\bibitem[Fal03]{F03}
Kenneth Falconer,
\newblock Fractal Geometry: Mathematical Foundations and Application. 
\newblock {\em John Wiley \& Sons}, 2003.

\bibitem[GKRS10]{Gawrychowski&Krieger&Rampersad&Shallit:2010} 
P. Gawrychowski, D. Krieger, N. Rampersad, and J. Shallit,
\newblock Finding the growth rate of a regular or context-free language in polynomial time.
\newblock \emph{Int. J. Found. Comput. Sci.} {\bf 21} (2010), 597--618.


\bibitem[GH11]{GH11}
A.~G\"{u}nayd{\i}n and P.~Hieronymi,
\newblock Dependent Pairs.
\newblock {\em J. Symbolic Logic}, 76(2): 377--390, 2011.

\bibitem[Haw20]{H20}
C.~Hawthorne,
\newblock Automata and tame expansions of $(\mathbb{Z},+)$.
\newblock  {\em Isr. J. Math.}, 249: 651--693, 2022.

\bibitem[HW19]{HW19}
P.~Hieronymi and E.~Walsberg,
\newblock Fractals and the monadic second order theory of one successor.
\newblock {\em J. Log. Anal.}, 15(5): 1--25, 2023.

\bibitem[IR86]{Ibarra&Ravikumar:1986} 
O. H. Ibarra and B. Ravikumar,
\newblock On sparseness, ambiguity, and other decision problems for acceptors and transducers. 
\newblock {\em STACS 86 (Orsay, 1986)}, 171--179, {\em Lecture Notes in Comput. Sci.}, {\bf 210}, \emph{Springer, Berlin}, 1986.

\bibitem[Kec95]{Kec} 
A. S. Kechris,
\newblock Classical descriptive set theory, Grad. Texts in Math., 156
\emph{Springer-Verlag}, New York, 1995.

\bibitem[Ked06]{K} 
K. S. Kedlaya,
\newblock Finite automata and algebraic extensions of function fields.
\newblock \emph{J. Th\'eorie Nombres Bordeaux}, {\bf 18} (2006), no. 2, 379--420.

\bibitem[Ked17]{K2} 
K. S. Kedlaya, 
\newblock On the algebraicity of generalized power series.
\newblock \emph{Beitr. Algebra Geom.}, {\bf 58} (2017), no. 3, 499--527. 

\bibitem[M02]{M02}
\newblock D. Marker, 
\newblock \emph{Model Theory: An Introduction}, 
\newblock Graduate Texts in Mathematics, Springer New York,
  2002.

\bibitem[Mi05]{Mi05}
C.~Miller,
\newblock Tameness in expansions of the real field.
\newblock In {\em Logic {C}olloquium `01}, volume~20 of {\em Lecture Notes in
  Logic}, pages 281--316. Association of Symbolic Logic, Urbana, IL, 2005.

\bibitem[MT18]{MT18}
C.~Miller and A.~Thamrongthanyalak,
\newblock D-minimal expansions of the real field have the zero set property.
\newblock {\em Proc. Amer. Math. Soc.}, 146: 5169--5179, 2018.

\bibitem[MT06]{MT06}
C.~Miller and J.~Tyne,
\newblock Expansions of o-Minimal Structures by Iteration Sequences.
\newblock {\em Notre Dame J. Formal Logic}, 47(1):93--99, 2006.

\bibitem[MS02]{MS02} 
R. Moosa and T. Scanlon,
\newblock The Mordell-Lang conjecture in positive characteristic revisited.
\newblock \emph{Model Theory and its applications}, 273--296, Quad. Mat. {\bf 11}, \emph{Aracne, Rome}, 2002.

\bibitem[MS04]{MS04}
R.~Moosa and T.~Scanlon,
\newblock $F$-structures and integral points on semiabelian varieties over finite fields.
\newblock \emph{Amer. J. Math.}, {\bf 126} (3):473--522, 2004.

\bibitem[Mor66]{M46}
P.~Moran,
\newblock Additive functions of intervals and Hausdorff measure.
\newblock \emph{Proc. Cambridge Philos. Soc.}, 42:15--23, 1946.

\bibitem[Lot02]{L02}
M.~Lothaire,
\newblock {\em Algebraic Combinatorics on Words}.
\newblock Cambridge University Press, Cambridge, UK, 2002.

\bibitem[Sch90]{Sch} H. P. Schlickewei, $S$-unit equations over number fields. \emph{Invent. Math.}, {\bf 102} (1990), 95--107.



\bibitem[Sim15]{S15}
Pierre Simon,
\newblock A guide to {NIP} theories.
\newblock {\em Lect. Notes Log.},
  vol.~44, Cambridge University Press, 2015.

\bibitem[Sip13]{S13}
M.~Sisper,
\newblock {\em Introduction to the Theory of Computation, 3rd edition}.
\newblock Cengage Learning. Cengage Learning India Private Limited, Delhi, India, 2013.

\bibitem[Tro81]{Trofimov:1981} 
V. I. Trofimov,
\newblock Growth functions of some classes of languages. (Russian).
\newblock \emph{Kibernetika (Kiev)} 1981, no. 6, i, 9--12, 149. 

\end{thebibliography}
\end{document}